\documentclass{elsarticle}
\makeatletter
\def\ps@pprintTitle{%
 \let\@oddhead\@empty
 \let\@evenhead\@empty
 \def\@oddfoot{\centerline{\thepage}}%
 \let\@evenfoot\@oddfoot}
\makeatother

\usepackage{amsfonts,amsmath,amsthm,amssymb,pstricks-add}
\newcommand{\Z}{\mathbb{Z}}
\newcommand{\G}{\Gamma}
\newcommand{\pres}[2]{\langle {#1}\ |\ {#2} \rangle}
\newcommand{\gpres}[1]{\langle {#1} \rangle}
\newcommand{\npres}[1]{\langle\langle {#1} \rangle\rangle}
\newtheorem{theorem}{Theorem}
\newtheorem{maintheorem}{Theorem}

\newtheorem{lemma}[theorem]{Lemma}
\newtheorem{corollary}[theorem]{Corollary}
\newtheorem{remark}[theorem]{Remark}
\newtheorem{example}[theorem]{Example}
\newtheorem{maincorollary}[maintheorem]{Corollary}
\newtheorem{definition}[theorem]{Definition}

\def\<{\langle}
\def\>{\rangle}
\def\cd{\mathrm{cd}\,}
\def\gd{\mathrm{gd}\,}
\def\dis{}
\def\ra{\rightarrow}

\begin{document}
\begin{frontmatter}
\title{Coherence, subgroup separability, and metacyclic structures for a class of cyclically presented groups}

\author[B]{W.A.Bogley}

\address[B]{Department of Mathematics, Kidder Hall 368, Oregon State University, Corvallis, OR 97331-4605, USA.}
\ead{Bill.Bogley@oregonstate.edu}
\author[W]{Gerald Williams}
\address[W]{Department of Mathematical Sciences, University of Essex, Wivenhoe Park, Colchester, Essex CO4 3SQ, UK.}
\ead{Gerald.Williams@essex.ac.uk}

\begin{abstract}
We study a class $\mathfrak{M}$ of cyclically presented groups that includes both finite and infinite groups and is defined by a certain combinatorial condition on the defining relations. This class includes many finite metacyclic generalized Fibonacci groups that have been previously identified in the literature. By analysing their shift extensions we show that the groups in the class $\mathfrak{M}$ are
are coherent, subgroup separable, satisfy the Tits alternative, possess finite index subgroups of geometric dimension at most two,
and that their finite subgroups are all metacyclic. Many of the groups in $\mathfrak{M}$ are virtually free, some are free products of metacyclic groups and free groups, and some have geometric dimension two. We classify the finite groups that occur in $\mathfrak{M}$, giving extensive details about the metacyclic structures that occur, and we use this to prove an earlier conjecture concerning cyclically presented groups in which the relators are positive words of length three. We show that any finite group in the class $\mathfrak{M}$ that has fixed point free shift automorphism must be cyclic.
\end{abstract}

\end{frontmatter}

\section{Introduction}\label{sec:intro}

Given a positive integer $n$, let $F$ be the free group with basis $x_0, \ldots, x_{n-1}$ and let $\theta: F \rightarrow F$ be the \em shift automorphism \em  given by $\theta(x_i) = x_{i+1}$ with subscripts modulo $n$. Given a word $w$ representing an element of $F$, the group
\[G_n(w) = \pres{x_0,\ldots, x_{n-1}}{w,\theta(w), \ldots, \theta^{n-1}(w)}\]
is called a \em cyclically presented group. \em The shift defines an automorphism of $G_n(w)$ with exponent $n$ and the resulting $\Z_n$-action on $G_n(w)$ determines the \em shift extension \em $E_n(w) = G_n(w) \rtimes_\theta \Z_n$, which admits a two-generator two-relator presentation of the form
\begin{equation}\label{eqn:ExtPres}
E_n(w) = \pres{t,x}{t^n,W}
\end{equation}
where $W = W(x,t)$ is obtained by rewriting $w$ in terms of the substitutions $x_i = t^ixt^{-i}$ (see, for example, \cite[Theorem~4]{JWW74}). The shift on $G_n(w)$ is then realized by conjugation by $t$ in $E_n(w)$. The group $G_n(w)$ is recovered as the kernel of a retraction $\nu^0: E_n(w) \rightarrow \Z_n = \pres{t}{t^n}$ that trivializes $x$.

In the fifty years since Conway's question in~\cite{Conway65}, the systematic study of cyclically presented groups has led to the introduction and study of numerous families of cyclic presentations with combinatorial structure determined by various parameters. These include the \em Fibonacci groups \em  $F(2,n) = G_n(x_0x_1x_2^{-1})$~\cite{Conway65} and a host of generalizing families including the groups $F(r,n)$~\cite{JWW74}, $F(r,n,k)$ \cite{CR74}, $R(r,n,k,h)$~\cite{CR75CMB}, $H(r,n,s)$~\cite{CR75LMS}, $F(r,n,k,s)$~\cite{CR79} (see also, for example,~\cite{Thomas91}), $P(r,n,l,s,f)$~\cite{Prishchepov95},\cite{Williams12}, $H(n,m)$~\cite{GilbertHowie95}, $G_n(m,k)=G_n(x_0x_mx_k^{-1})$~\cite{JM75,CHR98,BV03,Williams12b}, and the groups $G_n(x_0x_kx_l)$~\cite{CRS05,EW10,Bogley14}.

A number of these articles identify infinite families of cyclic presentations that define finite metacyclic groups. One purpose of this article is to give a unified explanation for results of this type that appeared in~\cite{CR75EMS},\cite{CR75CMB},\cite{CR75LMS},\cite{EW10},\cite{JWW74}, see Corollaries~\ref{cor:MetaCanadian}--\ref{cor:MetaFrnks} below. Our approach leads to the identification of new infinite families of cyclic presentations that define finite metacyclic groups, including but not limited to the class of presentations treated in~\cite{Prishchepov95}, see Corollary~\ref{cor:MetaPrisch} and Corollary~\ref{cor:Gshort}.

All of the cyclic presentations discussed in the preceding paragraph have a common combinatorial form. Given integers $f$ and $r \geq 0$, let
\begin{equation}
\Lambda(r,f) = \prod_{i=0}^{r - 1} x_{if}
\end{equation}
be the positive word of length $r$ consisting of equally spaced generators, starting with $x_0$ and having \em step size \em $f$ (where an empty product denotes the empty word). We call any shift of $\Lambda(r,f)$ or its inverse an \em $f$-block \em and we say that a cyclically presented group $G$ is of \em type \em $\mathfrak{F}$ if there is an integer $f$ such that $G \cong G_n(w)$ where $w$ is a product of two $f$-blocks. All of the cyclically presented groups mentioned in the previous paragraphs are of type~$\mathfrak{F}$. For example, any word $w$ of length three defines a group $G_n(w)$ of type~$\mathfrak{F}$. The groups $F(r,n)$ are defined by the word $w = x_0x_1\ldots x_{r-1}x_r^{-1}$, which is the product of two one-blocks having lengths $r\geq 1$ and one. The fact that every finite group is a quotient of a group of the form $F(r,n)$~\cite[Theorem~2]{JWW74} indicates the diversity of the class of groups of type~$\mathfrak{F}$ and the need to carefully identify subclasses of these groups in order to prove general results. A starting point is the following well-known combinatorial description of the shift extension.

\begin{lemma}\label{lemma:F}
A cyclically presented group $G$ is of type~$\mathfrak{F}$ if and only if there exist integer parameters $r,n,s,f,A,B$ such that $(r-s)f \equiv B-A$ mod $n$ and $G$ is isomorphic to the kernel of the retraction of the group

\begin{equation}\label{eqn:extF}
E = \pres{t,y}{t^n, y^rt^Ay^{-s}t^{-B}}.
\end{equation}
onto the cyclic subgroup of order $n$, generated by $t$, that is given by $\nu^f(t) = t$ and $\nu^f(y) = t^f$.
\end{lemma}
\begin{proof}
Working in the free group $F$ with basis $x_0, \ldots, x_{n-1}$, we are given that $G \cong G_n(w)$ where $w$ is the product of two $f$-blocks, so that $w = \theta^{a_1}(\Lambda(r_1,f))^{\delta} \cdot \theta^{a_2}(\Lambda(r_2,f))^{\epsilon}$ where $r_1, r_2 \geq 0$ and $\delta, \epsilon = \pm 1$. Up to cyclic permutation and inversion, we may assume that $a_1 = 0$ and $\delta = +1$. Setting $a = a_2$, the word $w$ has the form $w = \Lambda(r_1,f) \cdot \theta^{a}(\Lambda(r_2,f))^{\epsilon}$. Introducing $x,y,t$ and setting $x_i = t^ixt^{-i}$ and $y = xt^f$, the word $w$ transforms to $W = y^{r_1}t^{-r_1f} \cdot t^a(y^{r_2} t^{-r_2f})^\epsilon t^{-a}$ and the shift extension $G_n(w) \rtimes_\theta \Z_n$ admits a presentation $\pres{y,t}{t^n, W}$. If $\epsilon = +1$, then the first claim follows by setting $r = r_1$, $s = -r_2$, $A = a-r_1f$, and $B = a+r_2f$, while if $\epsilon = -1$, we set $r = r_1, s = r_2, A = a+(r_2-r_1)f$, and $B = a$. Conversely, given the retraction $\nu^f: E \ra \Z_n$, a standard Reidemeister-Schreier process (see e.g.~\cite[Theorem~2.3]{Bogley14}) shows that $\ker \nu^f \cong G_n(w)\cong G$ where $w$ is the product of two $f$-blocks.
\end{proof}

\begin{example}[Non-isomorphic cyclically presented groups sharing the same extension]\label{ex:commensurable}
\em The generalized Fibonacci group $F(3,6)$ is $G_6(x_0x_1x_2x_3^{-1})$ and has extension $$E = F(3,6) \rtimes_\theta \Z_6 \cong \pres{t,x}{t^6, xtxtxtx^{-1}t^{-3}} \stackrel{y=xt}{\cong} \pres{t,y}{t^6,y^3ty^{-1}t^{-3}}.$$ The extension $E$ admits retractions $\nu^f: E \rightarrow \Z_6 = \gpres{t}$ satisfying $\nu^f(t) = t$ and $\nu^f(y) = t^f$ when $2f-2\equiv 0$~mod~$6$, i.e.\,for $f = 1,4$~mod~$6$. As described at (\ref{eqn:extF}) in~\cite[Theorem 2.3]{Bogley14}, the kernel of $\nu^1$ is $F(3,6) = G_6(x_0x_1x_2x_3^{-1})$ with the generators $x_i = t^ixt^{-i} = t^iyt^{-1-i} \in E$. The kernel of $\nu^4$ is the cyclically presented group $G_6(y_0y_4y_2y_3^{-1})$ with generators $y_i = t^iyt^{-4-i} \in E$. Renaming the generators of this group as $u_i = y_{-i}$ with subscripts modulo $6$, the kernel of $\nu^4$ is isomorphic to the group $G_6(u_0u_2u_4u_3^{-1})$, which is the group $R(3,6,5,2)$ of~\cite{CR75CMB}. The group $E$ has order 9072, so as observed in~\cite{CR75CMB} the orders of $F(3,6)$ and $R(3,6,5,2)$ coincide, both being equal to~1512. The Sylow subgroup structure of the groups $F(3,6)$ and $R(3,6,5,2)$ are described in~\cite[page~165]{CR75ToddCox} and~\cite[page~175]{CR75CMB}, respectively. A calculation in GAP~\cite{GAP} shows that $F(3,6)$ is solvable with derived length 4, whereas $R(3,6,5,2)$ is solvable with derived length~3. Thus $F(3,6)\not \cong R(3,6,5,2)$. (For completeness we note that $E$ is solvable with derived length 4.)

Now the fact that
$$
\nu^4(x_i) = \nu^4(x) = \nu^4(yt^{-1}) = t^3 \in \pres{t}{t^6} \cong \Z_6
$$
implies that the intersection $\ker \nu^1 \cap \ker \nu^4$ has index two in $\ker \nu^1 \cong F(3,6)$. Thus the non-isomorphic groups $F(3,6)$ and $R(3,6,5,2)$ share the common extension $E$ and they meet in a common normal subgroup of index two within that extension.
\em\end{example}

Our main results involve a number of concepts pertaining to infinite groups. The following information on dimensions of groups may be found in~\cite{Brown10}. The \em geometric dimension \em ($\gd (G)$) of a group $G$ is the least nonnegative integer $d$ such that $G$ is the fundamental group of an aspherical CW complex of dimension~$d$. If no such $d$ exists, $\gd (G)=\infty$. The \em cohomological dimension \em ($\cd (G)$) of $G$ is the least nonnegative integer $d$ such that there exists a projective $\Z G$-resolution of the trivial module~$\Z$ of length at most $d$, or  $\cd (G)=\infty$ if no such $d$ exists. For any $G$ we have $\cd (G)\leq \gd (G)$ and if $H$ is a subgroup of $G$ then $\cd (H)\leq \cd (G)$ and $\gd (H)\leq \gd (G)$. We have $\cd (G)=0$ if and only if $\gd (G)=0$ if and only if $G=1$, and by a theorem of Stallings~\cite{Stallings68} and Swan~\cite{Swan69} we have $\cd (G)=1$ if and only if $\gd (G)=1$ if and only if $G$ is free and nontrivial (see~\cite[Example~2.4(b)]{Brown10}). Thus if $\gd (G)\leq 2$ then $\cd (G)=\gd (G)$. In cases where $\cd (G)=\gd (G)=d$, we say that $G$ is \em $d$-dimensional\em. Any group of finite cohomological dimension is torsion-free (see~\cite[Example~2.4(e)]{Brown10}) and a theorem of Serre (see~\cite[Theorem~4.4]{Brown10}) provides that if $H$ is a subgroup of finite index in a torsion-free group $G$, then $\cd (H) = \cd (G)$. Thus it follows that if $G$ possesses a finite index subgroup of geometric dimension at most two, then $G$ is virtually torsion-free, all torsion-free subgroups have geometric dimension at most two, and all torsion free subgroups of finite index have the same geometric dimension. Similarly, in a group that is \textit{virtually free} in the sense that it possesses a free subgroup of finite index, all torsion-free subgroups are free.

A group is \em coherent \em if each of its finitely generated subgroups is finitely presented. A group is \em subgroup separable \em or \em LERF \em if each of its finitely generated subgroups is the intersection of finite index subgroups. For finitely presented groups, subgroup separability implies that the membership problem is decidable for finitely generated subgroups~\cite{Malcev83}. In addition, any group that is subgroup separable is residually finite and so, if finitely generated, is Hopfian~\cite{Malcev83}. We make repeated use of the fact that if $H$ is a finite index subgroup of $K$, then $H$ is subgroup separable if and only if $K$ is subgroup separable~\cite[Lemma~1.1]{Scott78} (see also~\cite[Lemma~2.16]{Wise00}).

\begin{example}[Cyclically presented groups that are neither coherent nor subgroup separable]\label{ex:notsbgpsep}
\ \newline \em \noindent (i) Let $G=G_8(x_0x_1x_3)$, observing that $w=x_0x_1x_3$ is the product of two blocks with the same step size. A calculation in GAP~\cite{GAP} shows that the subgroup $H$ of $G$ generated by $$\{x_0x_4, x_1x_5, x_2x_6, x_3x_7,x_4x_0,x_5x_1, x_6x_2, x_7x_3\}$$ is normal, isomorphic to the right angled Artin group $F_4\times F_4$, and $G/H\cong \Z_3\times \Z_3$.
It follows that $G$ and its shift extension $E=\pres{y,t}{t^8,y^2tyt^4}$ are residually finite.  However, since
the direct product of free groups $F_n\times F_n$ ($n\geq 2$) contains a finitely generated subgroup that has undecidable membership problem~\cite{Mikhailova58} (or see~\cite[IV. Theorem~4.3]{LyndonSchupp}) it follows that neither $G$ nor its shift extension $E$ is subgroup separable, and since $F_n\times F_n$ ($n\geq 2$) is not coherent~\cite{Grunewald78}, neither $G$ nor $E$ is coherent. We thank Andrew Duncan for a helpful conversation that led to this observation. We record that in~\cite[page~774]{EW10} it was erroneously stated that $H$ is isomorphic to $\Z^8$.

\noindent(ii) Let $G=G_n(x_0x_1x_0^{-1}x_1^{-1})$ and note that $G$ is a right angled Artin group and that for $n=4$ we have that $G\cong F_2\times F_2$. Since right angled Artin groups are linear~\cite{Humphries94} (see also~\cite[Corollary~3.6]{HsuWise99}), the group~$G$ is residually finite. However, if $n\geq 4$ then $G$ is neither subgroup separable (by~\cite[Theorem~2]{MetaftsisRaptis08}) nor coherent (by~\cite[Theorem~1]{Droms87}). (If $n=2$ or $3$ then $G\cong \Z^n$ so it is both subgroup separable and coherent.) The extension $E=G\rtimes_\theta \Z_n=\pres{x,t}{t^n,xtxt^{-1}x^{-1}tx^{-1}t^{-1}}$ admits a retraction $\nu^f : E\rightarrow \Z_n=\gpres{t}$ satisfying $\nu^f(t)=t, \nu^f(x)=t^f$ for each $0\leq f\leq n-1$ with kernel $H=G_n(x_0x_{f+1}x_f^{-1}x_1^{-1})$. The conclusions for $G$, noted above, therefore also apply to $E$ and to $H$.
\em
\end{example}

We now introduce the groups that are the subject of this paper. We consider cyclically presented groups $G$ of type~$\mathfrak{F}$ that satisfy a certain condition relating the step size, block lengths, and signs occurring on the blocks that comprise the defining relator~$w$.

\begin{definition} \em A cyclically presented group $G$ is of \em type \em $\mathfrak{M}$ if $G \cong G_n(w)$ where $w$ is a product of two $f$-blocks having lengths $r_1$ and $r_2$; that is,
\begin{equation}
w = \Lambda(r_1,f) \cdot \theta^{a}(\Lambda(r_2,f))^{\epsilon},
\end{equation}
for some $0\le a\leq n-1$ where $(r_1+\epsilon r_2)f \equiv 0$ mod $n$.\em
\end{definition}
\noindent The cyclically presented groups of type~$\mathfrak{M}$ are those for which the shift extension, described in Lemma~\ref{lemma:F}, admits a particularly simple retraction onto the cyclic subgroup of order $n$ generated by $t$. This enables a uniform analysis of all of the cyclically presented groups of type~$\mathfrak{M}$.

\begin{lemma} A cyclically presented group $G$ is of type~$\mathfrak{M}$ if and only if there exist integer parameters $r,n,s,f,A$ such that $(r-s)f \equiv 0$ mod $n$ and $G$ is isomorphic to the kernel of the retraction of the group
\begin{alignat}{1}
E &= \pres{t,y}{t^n, y^rt^Ay^{-s}t^{-A}}.\label{eq:EwithB=A}
\end{alignat}
onto the cyclic subgroup of order $n$, generated by $t$, that is given by $\nu^f(t) = t$ and $\nu^f(y) = t^f$.
\end{lemma}

\begin{proof} Following the proof of Lemma~\ref{lemma:F}, the condition $(r_1+\epsilon r_2)f \equiv 0$~mod~$n$ is equivalent to the condition $(r-s)f \equiv 0$~mod~$n$ and implies that the parameters $A,B$ in~(\ref{eqn:extF}) satisfy $A \equiv B$~mod~$n$. The conclusions of the lemma follow routinely.
\end{proof}

Our strategy for identifying and analyzing the groups of type~$\mathfrak{M}$ is to work from the perspective of the shift extension $E$ and its retraction onto the cyclic group of order $n$, as determined by parameters $(r,n,s,f,A)$ where $(r-s)f \equiv 0$~mod~$n$.
To be explicit, the resulting group of type~$\mathfrak{M}$ has the form $G_n(w)$ where
\begin{alignat}{1}
w &= \begin{cases}
  \Lambda(r,f)\theta^A(\Lambda(s,f))^{-1} & \mathrm{if}~s\geq 0,\\
  \Lambda(r,f)\theta^{A+fr}(\Lambda(|s|,f))& \mathrm{if}~s\leq 0,
\end{cases}\nonumber\\
 &= \begin{cases}
		\prod_{i=0}^{r-1} x_{if}\left( \prod_{i=0}^{s-1}x_{A+if}\right)^{-1} & \mathrm{if}~s\geq 0,\\
		\prod_{i=0}^{r-1} x_{if}\prod_{i=0}^{|s|-1} x_{A+(r+i)f} & \mathrm{if}~s\leq 0.
\end{cases}\label{eqn:OurCyclicw}
\end{alignat}
\noindent We note that the condition $(r-s)f \equiv 0$ mod $n$ implies that $r-s$ is divisible by $n/(n,f)$ and that this latter condition plays a role in several of our arguments. The special nature of the shift extension for a group of type~$\mathfrak{M}$ means that for all choices of $(r,n,s,A)$, we can take $f = 0$ and so the cyclically presented group $G_n(x_0^rx_A^{-s})$ is of type~$\mathfrak{M}$. Indeed, every cyclically presented group of type~$\mathfrak{M}$ is commensurable with one of this form. As such, the following lemma plays a central role in all that follows.

\begin{lemma}[{\cite{Pride87},\cite[Lemma~3.4]{BW15}}]\label{lem:Pridescyclic}
If $(\rho,\sigma)=1$ and $m\geq 1$ then $G_m(x_0^\rho x_1^{-\sigma})\cong \Z_{|\rho^m-\sigma^m|}$. Any of the elements $x_i$ ($0\leq i<m$) can serve as a sole generator.
\end{lemma}

\begin{example}[Cyclically presented groups $G_n(x_0x_kx_l)$ that lie in~$\mathfrak{M}$]\label{ex:ConditionC}
\em
In~\cite{EW10}, the groups $G=G_n(x_0x_kx_l)$ were studied in terms of four conditions (A),(B),(C), (D) where, in particular:
\begin{itemize}
  \item[(A)] $n\equiv 0$~mod~$3$ and $k+l\equiv 0$~mod~$3$;
  \item[(C)] $3l \equiv 0$~mod~$n$ or $3k \equiv 0$~mod~$n$ or $3(l-k)\equiv 0$~mod~$n$.
\end{itemize}
Observe again that the defining word $w=x_0x_kx_l$ is a product of two blocks with the same step size.
By cyclically permuting the relators we have that $G_n(x_0x_kx_l)\cong G_n(w')\cong G_n(w'')\cong G_n(w''')$ where
\begin{alignat*}{1}
w'&=(x_0x_k)x_l=(x_0x_{f'})x_{2f'+A'}\\
w''&=(x_0x_{l-k})x_{n-k}=(x_0x_{f''})x_{2f''+A''}\\
w'''&=(x_0x_{n-l})x_{k-l}=(x_0x_{f'''})x_{2f'''+A'''}
\end{alignat*}
where $f'=k,f''=l-k,f'''=n-l$, $A'=l-2k,A''=k-2l,A'''=k+l$. Then each of $w',w'',w'''$ are of the form~(\ref{eqn:OurCyclicw}) where $r=2,s=-1$. With these values of $r,s$ the condition $f(r-s)\equiv 0$~mod~$n$ becomes $3f\equiv 0$~mod~$n$ and so the groups $G_n(w'),G_n(w''),G_n(w''')$ therefore lie in the class~$\mathfrak{M}$ when $3f'\equiv 0$, $3f''\equiv 0$, $3f'''\equiv 0$~mod~n (respectively). That is, if condition (C) holds then $G_n(x_0x_kx_l)$ lies in~$\mathfrak{M}$.
\em
\end{example}

As a corollary of one of our main results concerning groups in~$\mathfrak{M}$ we prove a conjecture and reprove a lemma from~\cite{EW10} that together deal with the groups $G_n(x_0x_kx_l)$ when condition~(C) holds.

Theorem~\ref{thm:Esubgroups} below  concerns the structure of the group $E$; part (d) should be compared with the fact that the class of Baumslag-Solitar groups $B(r,s) = \pres{t,y}{y^rty^{-s}t^{-1}}$ includes both non-Hopfian groups and non-residually finite groups \cite{BS62}. Recall that a group $G$ is \em metacyclic \em if it has a cyclic normal subgroup $K$ such that $G/K$ is cyclic. (Metacyclic groups have derived length~2, so the groups in Example~\ref{ex:commensurable} are not metacyclic.)

\begin{maintheorem}\label{thm:Esubgroups} Given integer parameters $(r,n,s,A)$ with $n > 0$, let $E =$ \linebreak $E(r,n,s,A)$ be the group with presentation $\pres{t,y}{t^n,y^rt^Ay^{-s}t^{-A}}$. Then:
\begin{enumerate}
\item[(a)] Each finite subgroup of $E$ is metacyclic;
\item[(b)] The group $E$ is virtually free if $r^{n/(n,A)} \neq s^{n/(n,A)}$ or if $r=s=0$;
\item[(c)] The group $E$ is finite if and only if $r^{n/(n,A)} \neq s^{n/(n,A)}$ and either
\begin{itemize}
\item[(i)] $(n,A) = (r,s) = 1$, in which case $E$ is metacyclic of order $n|r^n-s^n|$, or
\item[(ii)] $E = \gpres{t} \cong \Z_n$, which happens if and only if $|r-s| = 1$ and either $rs = 0$ or $A \equiv 0 \mod n$;
\end{itemize}
\item[(d)] The group $E$ is coherent and subgroup separable and possesses a finite index subgroup of geometric dimension at most two.
\end{enumerate}
\end{maintheorem}

Since $G = G_n(w)$ is a finite index subgroup of the shift extension $E = E(r,n,s,A)$, we have the following corollary:

\begin{maincorollary}\label{cor:Gshort}  Let $G$ be the cyclically presented group in class $\mathfrak{M}$ determined by the integer parameters $(r,n,s,f,A)$ where $f(r-s)\equiv 0$~mod~$n$. All of the claims (a)--(d) for the group $E$ in Theorem~\ref{thm:Esubgroups} descend directly to subgroup $G$ of index $n$, with the obvious modifications to the group orders in statement~(c).
\end{maincorollary}

More detailed analysis is given in Section~\ref{sec:MStructure}. For the case where $r^{n/(n,A)} \neq s^{n/(n,A)}$ Theorem~\ref{thm:ends} describes which values of the parameters produce virtually free groups that are non-elementary and which produce virtually infinite cyclic ones. Whenever $(n,A)$ or $(r,s) \neq 1$, the group $E$ in (\ref{eq:EwithB=A}) splits nontrivially as an amalgamated free product. Given $f$ for which $f(r-s) \equiv 0$~mod~$n$, the subgroup theorem for graphs of groups implies that if $G = G_n(w)$ where $w$ is given as in (\ref{eqn:OurCyclicw}), then $G$ is expressible as the fundamental group of a graph of groups. In practice one can work out the graph of groups description in full detail. We use these techniques in the proof of our second main theorem, Theorem~\ref{mainthm:Mgroups}, below. We also use them to show that if $(n,A)(n,gf) = n(n,A,gf)$, where $g = (r,s)$, then every finite subgroup of $G$ is cyclic (Theorem \ref{thm:onlyCyclic}). When $r^{n/(n,A)} = s^{n/(n,A)}$ Theorem~\ref{thm:2dim} shows that in many cases the group $G$ is torsion-free and has geometric dimension equal to two, so is not a free group. In Corollary~\ref{cor:Tits} we use Theorem~\ref{thm:Esubgroups} and Theorem~\ref{thm:2dim} to show that $E$ (and hence $G$) satisfies a strong form of the Tits alternative. Namely that it either contains a finite index subgroup that surjects onto the free group of rank~2 or is virtually abelian; a group satisfying the former property is said to be \em large\em. Theorem~3.1 of~\cite{Button10} shows that deficiency one LERF groups are either large or of the form $F_l\rtimes \Z$ (for some $l\geq 0$). Corollary~\ref{cor:Tits} can be viewed as a starting point for explorations into similar results for deficiency zero LERF groups and for cyclically presented LERF groups.

A group $L$ is metacyclic if and only if it has cyclic subgroups $S,K$ with $K$ normal in $L$ such that $L=SK$ (this is called a \em metacyclic factorisation \em of $L$). If $L$ has a metacyclic factorisation $L=SK$ with $S\cap K=1$ then $L$ is a \em split metacyclic group \em and $S$ is the \em complement \em of $K$. Any finite metacyclic group $L$ arises in a metacyclic extension $\Z_M \hookrightarrow L \twoheadrightarrow \Z_N$ and has a presentation of the form
\begin{equation}\label{eqn:BeylPres}
B(M,N,R,\lambda) = \pres{a,b}{a^M = 1, bab^{-1} = a^R, b^N = a^{\lambda \frac{M}{(M,R-1)}}}
\end{equation}
where $R^N \equiv 1$ modulo $M$. See~\cite[Chapter IV.2]{BeylTappe82} or~\cite[Chapter 3]{Johnson80}. In fact, one can further assume that $\lambda$ is a divisor of $(M,R-1)(M,1+R+ \cdots + R^{N-1})/M$~\cite[IV.2.8]{BeylTappe82}, in which case $H_2(L) \cong \Z_\lambda$~\cite{BeylJones74}. In particular, when $L$ admits a balanced (or cyclic) presentation, then $H_2(L) = 0$, so we may take $\lambda = 1$; in this case $L$ admits a 2-generator, 2-relator presentation~\cite[IV.2.9]{BeylTappe82}. When $L$ is split metacyclic one can also take $\lambda = 0$ (see, for example, \cite[(2.5)]{BeylTappe82}).

The next result describes the structure of the groups in~$\mathfrak{M}$ when $r^{n/(n,A)}\neq s^{n/(n,A)}$ and $(r,s)=1$ and so is a refinement of Corollary~\ref{cor:Gshort}(b).

\begin{maintheorem}\label{mainthm:Mgroups} Let $G$ be the cyclically presented group in the class $\mathfrak{M}$ determined by the integer parameters $(r,n,s,f,A)$ where $f(r-s) \equiv 0$~mod~$n$. Suppose that $(r,s) = 1$ and that $\mu = |r^{n/(n,A)}-s^{n/(n,A)}| \neq 0$. Let
\[\bar{G} = B\left(\frac{(n,f)\mu}{n}, \frac{n(n,A,f)}{(n,A)(n,f)}, (s\hat{r})^{f\bar{\alpha}/(n,A,f)},1\right)\]
where $r\hat{r} \equiv 1$~mod~$\mu$ and $\alpha\bar{\alpha} \equiv 1$~mod~$n/(n,A)$ where $\alpha = A/(n,A)$. Then $G$ is a free product of $(n,A,f)$ copies of $\bar{G}$ together with the free group of rank $k=(n,A)-(n,A,f)$:
\[ G \cong \left(\ast_{i=1}^{(n,A,f)} \bar{G}\right) \ast F_k.\]
In particular, if
\begin{equation}
(n,A)(n,f) = n(n,A,f)\label{eq:cyclicGbar}
\end{equation}
then
\[ G \cong  \left(\ast_{i=1}^{(n,A,f)} \Z_{(n,f)\mu/n}\right) \ast F_k\]
is a free product of cyclic groups.
\end{maintheorem}

The following corollary gives a complete identification of the structure of the groups $G_n(x_0x_kx_l)$ when condition (C) holds (see Example~\ref{ex:ConditionC}). In particular, the first case proves~\cite[Conjecture~3.4]{EW10}, the second case identifies the structure of the group $G_n(x_0x_kx_l)$ when $k=0$ or $l=0$ or $k=l$ given on~\cite[page~760]{EW10}, and the third case reproves~\cite[Lemma~2.5]{EW10}.

\begin{maincorollary}[{\cite[Lemma~2.5 and Conjecture~3.4]{EW10}}]\label{maincor:EWgroupswithC}
Suppose $(n,k,l)=1$ and condition~(C) holds and let $G=G_n(x_0x_kx_l)$. Then
\[
G\cong
\begin{cases}
B\left(\frac{2^n-(-1)^n}{3},3,2^{2n/3},1\right) & \mathrm{if~(A)~does~not~hold~and}~k\not \equiv 0,l\not \equiv 0,k\not \equiv l,\\
\Z_{2^n-(-1)^n} & \mathrm{if}~k\equiv 0,\mathrm{or}~l\equiv 0,\mathrm{or}~k\equiv l,\\
\Z_{({2^{n/3}-(-1)^{n/3}})/{3}}*\Z*\Z & \mathrm{if~(A)~holds}.
\end{cases}
\]
(where congruences are mod~$n$).
\end{maincorollary}

The following immediate corollary to Theorem~\ref{mainthm:Mgroups} describes the metacyclic structures of the nontrivial finite groups occurring in the class $\mathfrak{M}$.

\begin{maincorollary}\label{maincor:MgroupsCor}
Let $G$ be the cyclically presented group in the class $\mathfrak{M}$ determined by the integer parameters $(r,n,s,f,A)$ where $f(r-s)\equiv 0$~mod~$n$. Suppose that $(r,s) = (n,A)=1$ and that $\mu = |r^{n}-s^{n}| \neq 0$. Then $G$ is the finite metacyclic group
\[G\cong B\left( \frac{(n,f)(r^{n}-s^{n})}{n} , \frac{n}{(n,f)} , (s\hat{r})^{f\bar{A}},1 \right)\]
of order $|r^n-s^n|$ where $r\hat{r}\equiv 1$~mod~$|r^{n}-s^{n}|$ and $A\bar{A}\equiv 1$~mod~$n$.
\end{maincorollary}

Using Corollary~\ref{maincor:MgroupsCor} we also recover various other families of finite metacyclic cyclically presented groups from the literature and often provide additional detail about the metacyclic structure. Applying it to the Prischepov groups~\cite{Prishchepov95}
\[P(r,n,l,s,f)=G_n( (x_0x_f\ldots x_{(r-1)f}) (x_{(l-1)}x_{(l-1)+f} \ldots x_{(l-1)+(s-1)f} )^{-1})\]
(with $r,s>0$) we have (by replacing $A$ with $l-1$):

\begin{corollary}\label{cor:MetaPrisch}
Suppose $(n,l-1)=1$, $(r,s)=1$, $r\neq s$, $f(r-s)\equiv 0$~mod~$n$. Then $P(r,n,l,s,f)$ is the metacyclic group $B((n,f)(r^n-s^n)/n,n/(n,f),(s\hat{r})^{f\bar{A}},1)$ of order $|r^n-s^n|$, where $(l-1)\bar{A}\equiv 1$~mod~$n$ and $r\hat{r}\equiv 1$~mod~$(r^n-s^n)$.
\end{corollary}

In the case $s=1$ we have that the groups $P(r,n,l,s,k)$ coincide with the groups $R(r,n,k,h)$ of~\cite{CR75CMB}; specifically, $R(r,n,k,h)=P(r,n,(r-1)h+k+1,1,h)$.
Setting $s=1$, $f=h$, $l=(r-1)h+k+1$ in Corollary~\ref{cor:MetaPrisch} and noting that $r\cdot r^{n-1}\equiv 1$~mod~$(r^n-1)$ we get

\begin{corollary}\label{cor:MetaCanadian}
Suppose $(n,k)=1$, $r>1$, $h(r-1)\equiv 0$~mod~$n$. Then $R(r,n,k,h)$ is the metacyclic group $B((n,h)(r^n-1)/n,n/(n,h),r^{(n-1)h\bar{A}},1)$ of order $(r^n-1)$, where $((r-1)h+k)\bar{A}\equiv 1$~mod~$n$.
\end{corollary}
\noindent Under these hypotheses the groups $R(r,n,k,h)$ were shown to be metacyclic of order $(r^n-1)$ in the Corollary in~\cite{CR75CMB}. If additionally we have $h=1$ then we get the groups $F(r,n,k)$ of~\cite{CR75EMS}; specifically, $F(r,n,k)=P(r,n,r+k,1,1)$. Setting $h=1$ in Corollary~\ref{cor:MetaCanadian} we get

\begin{corollary}\label{cor:MetaFrnk}
Suppose $(n,k)=1$, $r>1$, $r\equiv 1$~mod~$n$. Then $F(r,n,k)$ is the metacyclic group $B((r^n-1)/n,n,r^{(n-1)\bar{A}},1)$ of order $(r^n-1)$, where $(r+k-1)\bar{A}\equiv 1$~mod~$n$.
\end{corollary}
\noindent Under these hypotheses the groups $F(r,n,k)$ were shown to be metacyclic of order $(r^n-1)$ in Theorem~1 of~\cite{CR75EMS} and the presentation $B((r^n-1)/n,n,r^{\bar{k}},1)$, where $k\bar{k}\equiv 1$~mod~$n$, was obtained.

In the case $f=1, l=r+1$ the groups $P(r,n,l,s,f)$ are the groups $H(r,n,s)$ of~\cite{CR75LMS}. Setting $f=1$, $l=r+1$ in Corollary~\ref{cor:MetaPrisch} we get
\begin{corollary}\label{cor:MetaHrns}
Suppose $(n,r)=1$, $(r,s)=1$, $r\neq s$, $r\equiv s$~mod~$n$. Then $H(r,n,s)$ is the metacyclic group $B((r^n-s^n)/n,n,(s\hat{r})^{\bar{A}},1)$ of order $|r^n-s^n|$, where $r\bar{A}\equiv 1$~mod~$n$ and $r\hat{r}\equiv 1$~mod~$(r^n-s^n)$.
\end{corollary}
\noindent Under these hypotheses Theorem~3(ii) of~\cite{CR75LMS} showed the groups $H(r,n,s)$ to be metacyclic of order $|r^n-s^n|$, being an extension of a cyclic group of order $|r^n-s^n|/n$ by a cyclic group of order~$n$.

In the case $f=1$ the groups $P(r,n,l,s,f)$ coincide with the groups $F(r,n,k,s)$ of~\cite{CR79}; specifically, $F(r,n,k,s)=P(r,n,r+k,s,1)$. Setting $f=1$, $l=r+k$ in Corollary~\ref{cor:MetaPrisch} we get

\begin{corollary}\label{cor:MetaFrnks}
Suppose $(n,r+k-1)=1$, $(r,s)=1$, $r\neq s$, $r\equiv s$~mod~$n$. Then $F(r,n,k,s)$ is the metacyclic group $B((r^n-s^n)/n,n,(s\hat{r})^{\bar{A}},1)$ of order $|r^n-s^n|$, where $(r+k-1)\bar{A}\equiv 1$~mod~$n$ and $r\hat{r}\equiv 1$~mod~$(r^n-s^n)$.
\end{corollary}
\noindent Under these hypotheses Theorem~2 of~\cite{CR79} showed the groups $F(r,n,k,s)$ to be metacyclic of order $|r^n-s^n|$ and the presentation $B((r^n-s^n)/n,n,\hat{s}r,1)$ for $F(r,n,k,s)$ was obtained, where $s\hat{s}\equiv 1$~mod~$(r^n-s^n)$.

A celebrated theorem of J.Thompson~\cite{Thompson59} holds that a finite group with a fixed point free automorphism of prime order must be nilpotent. Restricting attention to the shift automorphism operating on a finite cyclically presented group in the class $\mathfrak{M}$, the next result asserts, in particular, that if any power of the shift is fixed point free then $G$ is cyclic.

\begin{maintheorem}\label{thm:fpf}
Let $G$ be a finite and nontrivial group in class $\mathfrak{M}$ determined by the integer parameters $(r,n,s,f,A)$ where $f(r-s)\equiv 0$~mod~$n$, and let $0\leq j <n$. Then the fixed point subgroup $\mathrm{Fix}(\theta^j)$ has order $|r^{(n,j)}-s^{(n,j)}|$. In particular, the automorphism $\theta^j$ is fixed point free if and only if $(n,j)=|r-s|=1$, in which case $f\equiv 0$~mod~$n$ and $G$ is cyclic of order $|r^n-s^n|$.
\end{maintheorem}

The following notation will be used throughout this article:
\begin{equation}\label{eqn:parameters}
\left.
\begin{aligned}
    r&,n,s,f,A\in \Z, n>0,r\geq 0~\mathrm{where}~f(r-s)\equiv 0~\mathrm{mod}~n,\quad\\
    m&=n/(n,A),\ \alpha=A/(n,A),\\
    g&=(r,s),\ \rho =r/g,\ \sigma =s/g,~(\mathrm{and}~g=\rho=\sigma=0~\mathrm{if}~r=s=0),\\
    \mu&=|\rho^m-\sigma^m|,\\
    E&=E(r,n,s,A)=\pres{y,t}{t^n,y^rt^Ay^{-s}t^{-A}},\quad\\
    M&=M(r,n,s,A)=\pres{u,v}{v^m,u^\rho v^\alpha u^{-\sigma }v^{-\alpha}},\quad\\
    \nu^f&:E\rightarrow \pres{t}{t^n=1}~\mathrm{is~the~retraction~given~by}~\nu^f(y)=t^f, \nu^f(t)=t;\quad\\
    \mathfrak{M}&~\mathrm{is~the~class~of~groups}~G~\mathrm{where}\\
    G&= \mathrm{ker}\nu^f=G_n(w)~\mathrm{where}\\
    w&= \begin{cases}
		\prod_{i=0}^{r-1} x_{if}\left( \prod_{i=0}^{s-1}x_{A+if}\right)^{-1} & \mathrm{if}~s\geq 0,\\
		\prod_{i=0}^{r-1} x_{if}\prod_{i=0}^{|s|-1} x_{A+(r+i)f} & \mathrm{if}~s\leq 0.
\end{cases}\\
    \Gamma&=\mathrm{ker}\nu^0=\npres{y}_E=G_n(x_0^rx_A^{-s}),
\end{aligned}
\right\}
\end{equation}
where $\npres{y}_E$ denotes the normal closure of $y$ in $E$.

We conclude this introduction with an outline of the paper. In Section~\ref{sec:EProof} we prove Theorem~\ref{thm:Esubgroups} (and hence Corollary~\ref{cor:Gshort}) by expressing $E$ as an amalgamated free product of cyclic and metacyclic groups and carrying out a detailed study of this decomposition. In Section~\ref{sec:MStructure} we prove refinements of parts of Theorem~\ref{thm:Esubgroups} and Corollary~\ref{cor:Gshort} that apply to infinite groups $E$ and their cyclically presented subgroups $G$ and we prove the Tits alternative. In Section~\ref{sec:freeprodmeta} we study the graph of groups decomposition of $G$ and study the intersection of finite cyclically presented groups $\bar{G}$ and $\bar{\Gamma}$ lying within $G$ and $\Gamma$ to obtain a metacyclic presentation of $\bar{G}$ and prove Theorem~\ref{mainthm:Mgroups}. We then apply this to the groups $G_n(x_0x_kx_l)$ to prove Corollary~\ref{maincor:EWgroupswithC}. In Section~\ref{sec:FixedPoints}, by studying the shift dynamics of $\Gamma$, we obtain the necessary information about the shift dynamics of $G$ to prove Theorem~\ref{thm:fpf}. We close by comparing the finite groups $E$ and $G$ obtained in this article with the finite groups $J_n(m,k)$ and their cyclically presented subgroups considered in~\cite{BW15}.

\section{The Structure of the Group $E$}\label{sec:EProof}

Our objective in this section is to prove Theorem~\ref{thm:Esubgroups}. Consider the group $E_n(w)$ in~(\ref{eqn:ExtPres}), and write
\[W=W(x,t)=x^{r_1}t^{\beta_1}\ldots x^{r_k}t^{\beta_k}\]
where $k\geq 1$ and each $r_i\in \Z\backslash \{0\}$, $1\leq \beta_i\leq n-1$. Set $h=(r_1,\ldots,r_k)$, $b=(n,\beta_1,\ldots ,\beta_k)$ and form the group
$$
N = \pres{u,v}{v^{n/b},W'(u,v)}
$$
where
\[W'(u,v)=u^{r_1/h}v^{\beta_1/b}\ldots u^{r_k/h}v^{\beta_k/b}.\]
Since $t$ has order $n$ in $E_n(w)$, it follows that with the identifications $v = t^b$ and $u = x^h$, the group $E_n(w)$ splits as an amalgamated free product
\begin{alignat}{1}
E_n(w) &=\pres{x,t}{t^n,W(x,t)}\nonumber\\
&\cong \pres{t}{t^n} \ast_{t^b = v} \pres{u,v}{v^{n/b}, W'(u,v)} \ast_{u = x^h} \gpres{x} \nonumber\\
&\cong \Z_n \ast_{\Z_{n/b}} N\ast_{u=x^h} \gpres{x}\label{eqn:longEsplit}
\end{alignat}
where the cyclic groups $\gpres{u}$ and $\gpres{v}$ can be finite or infinite.

In this notation we have
\begin{lemma}\label{lem:treeofGpsE}
\begin{itemize}
  \item[(a)] Let $P$ be any property of groups that is preserved by taking subgroups and is satisfied by all finite cyclic groups. If $N$ satisfies $P$ then every finite subgroup of $E_n(w)$ satisfies $P$.
  \item[(b)] If $N$ is finite then $E_n(w)$ is virtually free.
\end{itemize}
\end{lemma}

\begin{proof}
(a) Using the free product decomposition~(\ref{eqn:longEsplit}), \cite[I.7.11]{DD89} implies that each finite subgroup $H$ of $E$ is conjugate to a subgroup of either $\gpres{t},N,$ or $\gpres{x}$. Since the property $P$ is preserved by taking subgroups and all finite cyclic groups satisfy $P$, the result follows.

(b) If $N$ is finite then $\gpres{x}$ is finite so~(\ref{eqn:longEsplit})  gives that $E_n(w)$ is the fundamental group of a (finite) graph of finite groups so $E_n(w)$ is virtually free by the (easier part) of a theorem of Stallings~\cite[IV.1.6]{DD89}.
\end{proof}

We now concentrate on the case where $E = \pres{t,y}{t^n, y^rt^Ay^{-s}t^{-A}}$ as in~(\ref{eqn:parameters}) and form the group
$$
M = \pres{u,v}{v^{n/(n,A)},u^{\rho}v^\alpha u^{-\sigma}v^{-\alpha}}.
$$
where $\alpha=A/(n,A)$. We show in Lemma \ref{lemma:M} below that the group $M$ is split metacyclic of the form $M = \gpres{u} \rtimes \gpres{v}$ where $v$ has order $n/(n,A)$ and the order of $u$ can be finite or infinite. It follows that with the identifications $v = t^{(n,A)}$ and $u = y^g$, the group $E$ splits as an amalgamated free product
\begin{alignat}{1}
E &\cong \pres{t}{t^n} \ast_{t^{(n,A)} = v} \pres{u,v}{v^{n/(n,A)}, u^{\rho}v^{\alpha}u^{-\sigma}v^{-\alpha}} \ast_{u = y^g} \gpres{y}\nonumber\\
 &\cong \Z_n \ast_{\Z_{n/(n,A)}} M \ast_{u=y^g} \gpres{y}\label{eqn:Esplit}
\end{alignat}
where the cyclic groups $\gpres{u}$ and $\gpres{y}$ can be finite or infinite.

In the statement of the following lemma the parameters $m,\rho,\sigma,\alpha$ are any integers (with $m\geq 1$); in our applications these parameters will take the roles specified in~(\ref{eqn:parameters}).

\begin{lemma}\label{lemma:M} Given integer parameters $m, \rho, \sigma,\alpha$ with $m \geq 1$, $(\rho,\sigma)=1$, $(m,\alpha)=1$ the group $M$ with presentation $M = \pres{u,v}{v^m,u^\rho v^\alpha u^{-\sigma}v^{-\alpha}}$ is split metacyclic of the form $\gpres{u} \rtimes \gpres{v}$
with cyclic normal subgroup $\gpres{u} \cong \Z_{|\rho^m-\sigma^m|}$ and complement $\gpres{v} \cong \Z_m$.
Furthermore, in this situation we have the following:
\begin{enumerate}
\item[(a)] The group $M$ is finite if and only if $\rho^m \neq \sigma^m$;
\item[(b)] If $\rho^m = \sigma^m$, then either $\rho = \sigma = \pm 1$, in which case $M \cong \Z \times \Z_m$, or else $\rho = -\sigma = \pm 1$ and $m$ is even, in which case $M \cong \Z \rtimes \Z_m$;
\item[(c)] The group $M$ admits a presentation of the form $B(\rho^m-\sigma^m,m,\rho\hat{\sigma},0)$ where $\sigma\hat{\sigma} \equiv 1$~mod~$ \rho^m-\sigma^m$.
\end{enumerate}
\end{lemma}

\begin{proof}
Using a retraction $\xi^0: M \rightarrow \gpres{v} \cong \Z_m$ with $\xi^0(v) = v$ and $\xi^0(u) = v^0 = 1$, we have $\ker \xi^0 = \gpres{u} \cong G_m(u_0^{\rho}u_1^{ -\sigma}) \cong \Z_{|\rho^m-\sigma^m|}$, by Lemma~\ref{lem:Pridescyclic}, where $u_i = v^iuv^{-i} \in M$. In detail,
\begin{alignat*}{1}
u^{\rho^m} &= (u^\rho)^{\rho^{m-1}} = (v^\alpha u^\sigma v^{-\alpha})^{\rho^{m-1}} = (v^\alpha u^\rho v^{-\alpha })^{\sigma\rho^{m-2}}\\
&\qquad \qquad \qquad = (v^{2\alpha } u^\sigma v^{-2\alpha })^{\sigma\rho^{m-2}} = \cdots = v^{m\alpha } u^{\sigma^m}v^{-m\alpha } = u^{\sigma^m}.
\end{alignat*}
Using the fact that $(\rho,\sigma) = 1$ we can choose $\hat{\sigma}$ with $\sigma\hat{\sigma} \equiv 1$~mod~$\rho^m-\sigma^m$. Note that if $\rho^m = \sigma^m$ then the fact that $(\rho,\sigma) = 1$ implies that either $\rho = \sigma = \pm 1$ or else $\rho = -\sigma = \pm 1$ and $m$ is even, in which case we have $\hat{\sigma} = \sigma = \pm 1$. Now the fact that $u^{\rho^m-\sigma^m} = 1$ implies $u = u^{\sigma\hat{\sigma}}$ and we have
$$
v^\alpha uv^{-\alpha } = v^\alpha u^{\sigma\hat{\sigma}}v^{-\alpha } = (v^\alpha u^\sigma v^{-\alpha })^{\hat{\sigma}} = u^{\rho\hat{\sigma}}.
$$
That $M \cong \gpres{u} \rtimes \gpres{v}$ has the presentation
\[\pres{u,v}{u^{\rho^m-\sigma^m}=1, v^\alpha uv^{-\alpha} = u^{\rho\hat{\sigma}}, v^m=1}\]
follows easily and since $(m,\alpha)=1$ we may apply an automorphism of $\Z_m$ to obtain the presentation $B(\rho^m-\sigma^m,m,\rho\hat{\sigma},0)$ in part~(c). This in turn implies that $M$ is finite if and only if $u$ has finite order, which occurs precisely when $\rho^m \neq \sigma^m$, as in~(a). As above, the claims in~(b) follow routinely by taking account of the fact that $(\rho,\sigma) = 1$.
\end{proof}

For the proof of Theorem~\ref{thm:Esubgroups} we will also need the following.

\begin{lemma}\label{lemma:tCapy} The group $E$ is a semidirect product $E = \G \rtimes \gpres{t}$, where $\G = \npres{y}_E$. In particular $\G \cap \gpres{t} = 1$. Furthermore, if $y = 1$ in $E$, then $|r-s| = 1$, so $(r,s) = 1$.
\end{lemma}

\begin{proof}
The first claim follows from the fact that $\npres{y}_E$ is the kernel of a retraction of $E$ onto the cyclic subgroup generated by $\gpres{t}$:
\[E/\npres{y}_E \cong \pres{t,y}{t^n,y^rt^Ay^{-s}t^{-A},y} \cong \pres{t}{t^n}.\]
If $y=1$ in $E$, then the quotient $E/\npres{t}_E \cong \pres{t,y}{t^n,y^rt^Ay^{-s}t^{-A},t} \cong \pres{y}{y^{r-s}} \cong \Z_{|r-s|}$ obtained by trivializing $t$ is the trivial group, so $(r,s)$ is a divisor of $|r-s| = 1$.
\end{proof}

\begin{proof}[Proof of Theorem~\ref{thm:Esubgroups}]
If $r = s = 0$, then $E = \pres{t,y}{t^n,y^rt^Ay^{-s}t^{-A}}$ is the free product $E \cong \Z_n \ast \Z$, which is virtually free, so the conclusions of the theorem are easily verified in this case. So now we assume that either $r$ or $s\neq0$, so $g = (r,s) \neq 0$ and we have the tree of groups decomposition~(\ref{eqn:Esplit}). By Lemma~\ref{lemma:M} we have that $M$ is a metacyclic group, and since subgroups of metacyclic groups are metacyclic,  conclusion~(a) follows from Lemma~\ref{lem:treeofGpsE}(a). Recall that $m=n/(n,A)$. If $r^m-s^m\neq 0$ then by Lemma~\ref{lemma:M} we have that $M$ is finite, so conclusion~(b) follows from Lemma~\ref{lem:treeofGpsE}(b). If $r^{m}-s^{m}=0$ then Lemma~\ref{lemma:M} gives that $M$ is infinite and since $M$ embeds in $E$ (by~\cite[I.7.5]{DD89}) we have that $E$ is infinite. Suppose then that $E$ is finite. By the same reasoning it follows from Lemma~\ref{lemma:M} that $r^{m}-s^{m}\neq 0$.

Still assuming that $E$ is finite, it follows (from~\cite[I.7.11]{DD89}) that $E$ is equal to one of its subgroups $\<t\>$, $M$, or $\<y\>$. Because $\<t\> \cong \Z_n$ is nontrivial and $\<t\> \cap \<y\> = 1$ by Lemma~\ref{lemma:tCapy}, it follows that $E \neq \<y\>$, so $E = \<t\>$ or $M$. If $E = M$ then $t \in \<v\> = \<t^\alpha\>$ and $y \in \<u\> = \<y^g\>$, so $a = (n,A) = 1$ and $g = (r,s) = 1$, as claimed. So assume that $E = \<t\> \cong \Z_n$, in which case Lemma~\ref{lemma:tCapy} provides that $y \in \<t\> \cap \<y\> = 1$, so that $y=1$ in $E$ and hence $|r-s| = g = 1$. The fact that $y=1$ implies that $u =y^g = 1$ and so Lemma~$\ref{lemma:M}$ implies that $\dis |r^{m} - s^{m}| = \left|\left(\frac{r}{g}\right)^{m} - \left(\frac{s}{g}\right)^{m}\right| = 1$. If we now suppose that $rs \neq 0$, then $\min\{|r|,|s|\} \geq 1$, so the Mean Value Theorem implies that $\dis 1 = |r^{m} - s^{m}| = m|x|^{m -1}|r-s| = m|x|^{m-1}$ for some real number $x$ satisfying $|x| \geq 1$. Thus $m = 1$ and so $A \equiv 0 \mod n$. Using Lemma~\ref{lemma:M} as needed, the remaining components of the statement (c) are easily verified.

It remains to verify the global properties expressed in statement~(d). By~(\ref{eqn:Esplit}) and Lemma~\ref{lemma:M} the group $E$ is the fundamental group of a tree of groups where the edge groups are cyclic and the vertex groups are either cyclic or metacyclic, hence coherent. Then~\cite[Lemma~4.8]{Wilton09} provides that $E$ is coherent. If $r^{m} \neq s^{m}$ then $E$ is virtually free by~(b). Since free groups are subgroup separable~\cite{Hall49} and have geometric dimension  at most one, we can assume that $r^{m} = s^{m}$. With this, Lemma~\ref{lemma:M} implies that $\gpres{u}$ and $\gpres{y}$ are infinite cyclic and that $M$ is split metacyclic of the form $M = \gpres{u} \rtimes \gpres{v}$. Again, with $\G = \npres{y}_E $, any element of $\G \cap M$ has the form $u^iv^j$ with $1 = \nu^0(u^iv^j) = v^j$ so $j \equiv 0$~mod~$m$ and therefore $u^iv^j = u^i \in \gpres{u}$. Thus $\G \cap M$ is a subgroup of the infinite cyclic group $\gpres{u}$. Since $\G \cap \gpres{t}$ is trivial, the subgroup theorem for free products with amalgamation~\cite{KarrassSolitar70} (or the structure theorem for groups acting on trees~\cite[I.4.1]{DD89}) implies that $\G$ is the fundamental group of a graph of groups with vertex groups that are either trivial or infinite cyclic. That any such group is the fundamental group of an aspherical two-complex follows from~\cite[Theorem 4.6]{CCH81}. Thus $E$ contains the finite index subgroup $\Gamma$ with  geometric dimension at most two.

To prove that $E$ is subgroup separable under the assumption that $r^{m} = s^{m}$, it suffices to prove that the finite index subgroup $\G=G_n(x_0^rx_A^{-s})$ is subgroup separable. By~\cite[Theorem 1]{Tretkoff90} it suffices to show that $\G$ is the fundamental group of a finite tree of groups where each vertex group is either trivial or infinite cyclic. The group $\G$ decomposes as a free product $\G \cong \ast_{i=1}^{(n,A)} G_{m}(x_0^rx_{\alpha}^{-s})$. Letting $\bar{\alpha}$ be a multiplicative inverse to $\alpha$ modulo $m$, the assignment $x_i \rightarrow x_{\bar{\alpha} i}$ determines an isomorphism $G_{m}(x_0^rx_{\alpha}^{-s}) \cong G_{m}(x_0^rx_1^{-s})$. Letting $i$ range from $0$ to $m-1$ with subscripts modulo $m$, we have
\begin{alignat*}{1}
G_{m}(x_0^rx_1^{-s}) &\cong \pres{x_i}{x_i^{r} = x_{i+1}^{-s}}\\
&\cong \pres{x_i,y_i}{x_i^{r} = x_{i+1}^{-s},y_i=x_i^g}\\
&\cong \pres{x_i,y_i}{y_i^{\rho} = y_{i+1}^{-\sigma},y_i=x_i^g}.
\end{alignat*}
Using the fact that $r^{m}=s^{m}$, it follows that either $r=s=g$ or else $r=-s=g$ and $m$ is even, so by Lemma~\ref{lem:Pridescyclic} we have that $G_{m}(y_0^{\rho} y_1^{-\sigma}) \cong \Z$. Thus the group $G_{m}(x_0^rx_1^{-s})$ is obtained from $G_{m}(y_0^{\rho} y_1^{-\sigma}) \cong \Z = \gpres{y_0} = \gpres{y_1} = \cdots = \gpres{y_{m-1}}=Y$, say, by adjoining a $g$th root of $y_i$ for each $i = 0,\ldots, m-1$ and so can be expressed as the fundamental group of a finite tree of groups with infinite cyclic vertex groups, as in Figure~\ref{fig:star}. Therefore $\Gamma$ can also be expressed as the fundamental group of a finite tree of groups with all vertex groups trivial or infinite cyclic, as desired.
\end{proof}

\begin{figure}
\psset{xunit=0.8cm,yunit=0.8cm,algebraic=true,dimen=middle,dotstyle=o,dotsize=3pt 0,linewidth=0.8pt,arrowsize=3pt 2,arrowinset=0.25}
\begin{center}
\vspace{-2cm}
\begin{pspicture*}(-7.8,-10.94)(23.78,6.3)
\psline(0.,0.)(3.,0.)
\psline(0.,0.)(2.567486427580694,1.5517775111751448)
\psline(0.,0.)(1.0059739965799328,2.8263078951531444)
\psline(0.,0.)(-1.0338243139016443,2.8162399201747346)
\psline(0.,0.)(2.603505526245509,-1.490556598995521)
\begin{scriptsize}
\psdots[dotstyle=*](0.,0.)
\rput[bl](-0.42,-0.4){$Y$}
\psdots[dotstyle=*](3.,0.)
\rput[bl](3.16,-0.28){$\langle x_0 \rangle$}
\rput[bl](1.5,-0.48){$\langle y_0 \rangle$}
\psdots[dotstyle=*](2.567486427580694,1.5517775111751448)
\rput[bl](2.7,1.44){$\langle x_1 \rangle$}
\rput[bl](1.44,0.52){$\langle y_1 \rangle$}
\psdots[dotstyle=*](1.0059739965799328,2.8263078951531444)
\rput[bl](0.8,2.98){$\langle x_2 \rangle$}
\rput[bl](0.64,1.32){$\langle y_2 \rangle$}
\psdots[dotstyle=*](-1.0338243139016443,2.8162399201747346)
\rput[bl](-1.44,3.){$\langle x_3 \rangle$}
\rput[bl](-0.56,1.48){$\langle y_3 \rangle$}
\psdots[dotstyle=*](2.603505526245509,-1.490556598995521)
\rput[bl](2.78,-1.7){$\langle x_{m-1} \rangle$}
\rput[bl](0.68,-1.52){$\langle y_{m-1}\rangle$}
\psdots[dotsize=1pt 0,dotstyle=*](-1.,0.42)
\psdots[dotsize=1pt 0,dotstyle=*](-1.,0.)
\psdots[dotsize=1pt 0,dotstyle=*](-0.9,-0.38)
\psdots[dotsize=1pt 0,dotstyle=*](-0.7,-0.62)
\end{scriptsize}
\end{pspicture*}
\end{center}
\vspace{-7.5cm}
\caption{Graph of groups for the group $G_{m}(x_0^{r}x_1^{-s}$).\label{fig:star}}
\end{figure}
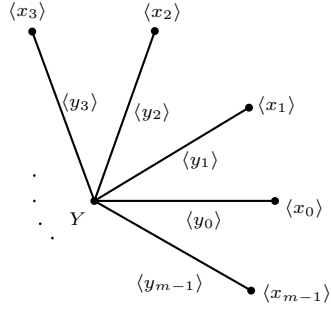

\begin{remark}
\em With $r^m=s^m$, the group $G_m(x_0^rx_1^{-s})$ appearing in the proof of Theorem~\ref{thm:Esubgroups} is a unimodular generalized Baumslag-Solitar group (see~\cite{Levitt07} for definitions) and so by~\cite[Proposition~2.6]{Levitt07} it is virtually $F_k\times \Z$ for some $k>1$, where $F_k$ denotes the free group of rank~$k$. Therefore, if $(n,A)=1$ then $E$ is virtually $F_k\times \Z$ if $r^n=s^n\neq 0$ and is virtually free otherwise. We thank Jack Button for pointing this out to us.\em
\end{remark}

\section{The Structure of Infinite Groups in the Class~$\mathfrak{M}$}\label{sec:MStructure}

In the notation~(\ref{eqn:parameters}) the groups $E$ and $G$ are finitely generated and when $\mu \neq 0$ they are virtually free by part~(b) of Theorem~\ref{thm:Esubgroups} and Corollary~\ref{cor:Gshort}. We can calculate the rational Euler characteristics~(\cite{Wall61}) $\chi(E)$ and $\chi(G) = n\chi(E)$ in terms of the amalgamated free product description (\ref{eqn:Esplit}) and these invariants relate directly to the rank of any finite index free subgroup as in~\cite[IV.1.10(1)]{DD89}. In particular, $E$ and $G$ possess nonabelian free subgroups precisely when $\chi(E) < 0$. The groups $E$ and $G$ are virtually infinite cyclic in the case when $\chi(E) = \chi(G) = 0$. The groups $E$ and $G$ are finite when $\chi(E) > 0$. The following is a refinement of parts (b) and~(c) of Theorem~\ref{thm:Esubgroups} and Corollary~\ref{cor:Gshort}.

\begin{theorem}[Virtually Free Case]\label{thm:ends} With the notation (\ref{eqn:parameters}), assume that $r^{n/(n,A)} \neq s^{n/(n,A)}$.
\begin{enumerate}
\item[(a)] The groups $G$ and $E$ are virtually infinite cyclic if and only if one of the following holds:
\begin{itemize}
\item[(i)] $\frac{1}{n} + \frac{1}{(r,s)} = 1$ and either $(n,A) = 1$ or $rs = 0$, or
\item[(ii)] $\frac{1}{n} + \frac{1}{|r-s|} = 1$ and $A \equiv 0$~mod~$n$.
\end{itemize}
\item[(b)] The groups $G$ and $E$ are virtually nonabelian free if and only if one of the following holds:
\begin{itemize}
\item[(i)] $\frac{1}{n} + \frac{1}{(r,s)} < 1$ and either $(n,A) = 1$ or $rs = 0$,
\item[(ii)] $\frac{1}{n} + \frac{1}{|r-s|} < 1$ and $A \equiv 0$~mod~$n$, or
\item[(iii)] $A \not \equiv 0$~mod~$n$, $rs \neq 0$, and $(n,A) \geq 2$.
\end{itemize}
\item[(c)] The groups $G$ and $E$ are finite in all other cases.
 \end{enumerate}
\end{theorem}

\begin{proof}
Using~(\ref{eqn:Esplit}) we have
\begin{alignat*}{1}
\chi(E) &= \frac{1}{|\<t\>|} - \frac{1}{|\<v\>|} + \frac{1}{|M|} - \frac{1}{|\<u\>|} + \frac{1}{|\<y\>|}\\
&= \frac{1}{n} - \frac{1}{m} + \frac{1}{m\mu} - \frac{1}{\mu} + \frac{1}{g\mu}\qquad \mathrm{(using~Lemma~\ref{lemma:M})}\\
&= \frac{1-(n,A)}{n}+\frac{1}{\mu}\left(\frac{(n,A)}{n} + \frac{1}{g} - 1\right).
\end{alignat*}
As a first case, suppose that $rs = 0$. Since $\mu \neq 0$, it follows that exactly one of $r$ or $s$ is equal to $0$ and this implies that $\mu = 1$. So in this case we have
$$
\chi(E) = \frac{1-(n,A)}{n}+\left(\frac{(n,A)}{n} + \frac{1}{g} - 1\right) = \frac{1}{n}+\frac{1}{g} - 1.
$$
Thus the claims are true when $rs = 0$. Next suppose that $(n,A) = 1$. Then
$$
\chi(E) = \frac{1}{\mu}\left(\frac{1}{n} + \frac{1}{g} - 1\right)
$$
and so the claims are true when $(n,A) = 1$. If $A \equiv 0$~mod~$n$ then, in particular, $m=1$ and so
$$
\chi(E) = \frac{1-n}{n}+\frac{g}{|r-s|}\left(\frac{n}{n} + \frac{1}{g} - 1\right) = \frac{1}{n} +\frac{1}{|r-s|} - 1
$$
and so the claims are true when $A \equiv 0$~mod~$n$. We can now assume that $rs \neq 0, A \not \equiv 0$~mod~$n$, and $(n,A) \geq 2$. Under the first two of these assumptions, together with the fact that $\mu \neq 0$, we will show that $\mu \geq 3$. With this, using the fact that $g \geq 1$ and $(n,A) \geq 2$, we have
\begin{alignat*}{1}
\chi(E) &= \frac{1-(n,A)}{n}+\frac{1}{\mu}\left(\frac{(n,A)}{n} + \frac{1}{g} - 1\right)\\
&\leq \frac{1-(n,A)}{n}+\frac{1}{\mu}\left(\frac{(n,A)}{n}\right)\\
&= \frac{1}{n}\left( 1-(n,A)+\frac{(n,A)}{\mu}\right)
\end{alignat*}
which is negative whenever $\mu\geq 3$. We complete the proof by showing that $\mu\geq 3$ in all cases. Since $A \not \equiv 0$~mod~$n$, we have that $m \geq 2$. Now, using the Mean Value Theorem, for some real number $x$ satisfying $\dis |x| \geq \min\left\{\frac{|r|}{g},\frac{|s|}{g}\right\}$, we have
\begin{alignat}{1}
\mu = \left|\left(\frac{r}{g}\right)^{m}-\left(\frac{s}{g}\right)^{m}\right|
\geq  \left|\left(\frac{|r|}{g}\right)^{m}-\left(\frac{|s|}{g}\right)^{m}\right|
= m|x|^{m-1}\left|\frac{|r|}{g}-\frac{|s|}{g}\right|.\label{eq:muineq}
\end{alignat}
Now the fact that $rs \neq 0$ implies that $|x| \geq 1$ and the fact that $\mu \neq 0$ implies that $|r|/g \neq |s|/g$, and so the integer $ \dis \left|\frac{|r|}{g}-\frac{|s|}{g}\right| \geq 1$. Thus we have $\mu \geq m$ and so the claim is proved unless $m = 2$. If $m=2$ then~(\ref{eq:muineq}) gives
$$
\mu = \left|\left(\frac{r}{g}\right)^{2}-\left(\frac{s}{g}\right)^{2}\right| \geq 2|x|\left|\frac{|r|}{g}-\frac{|s|}{g}\right|.
$$
The rightmost part of this expression is at least 4, except possibly when \linebreak $\left\{|r|/g,|s|/g\right\} = \{1,2\}$, in which case the middle part of this expression is equal to 3.
In all cases we are assured that $\mu \geq 3$.
\end{proof}

Next, in Theorem~\ref{thm:onlyCyclic}, we offer a refinement of Corollary~\ref{cor:Gshort}(a), which deals with the finite subgroups of groups from the class $\mathfrak{M}$. We use the splitting $E \cong \pres{t}{t^n} \ast_{t^{(n,A)} = v} M \ast_{u = y^g} \gpres{y}$ as in (\ref{eqn:Esplit}) where $M = \pres{u,v}{v^{m},u^{\rho}v^\alpha u^{-\sigma}v^{-\alpha}} \cong \gpres{u} \rtimes \gpres{v} \cong \Z_\mu \rtimes \Z_{m}$ by Lemma \ref{lemma:M}.

\begin{theorem}[Only Cyclic Finite Subgroups]\label{thm:onlyCyclic} With the notation (\ref{eqn:parameters}), if the condition
\begin{equation}\label{eqn:cyclicCondition}
(n,A)(n,gf) = n(n,A,gf)
\end{equation}
is satisfied, then each finite subgroup of $G$ is isomorphic to a subgroup of $\Z_l$ where
\[ l=\frac{|r^{n/(n,A)}-s^{n/(n,A)})|(n,f)}{(r,s)^{n/(n,A)-1}n}= \frac{g\mu(n,f)}{n}.\]
In particular, if $r^{n/(n,A)} = s^{n/(n,A)}$ and the condition (\ref{eqn:cyclicCondition}) is satisfied, then $G$ is torsion-free so has geometric dimension at most~2.
\end{theorem}

\begin{proof}
We first show that the intersection $G \cap M$ is contained in the cyclic normal subgroup $\gpres{u} \cong \Z_\mu$ of $M$ if and only if~(\ref{eqn:cyclicCondition}) is satisfied. We have $\nu^f(v)= v = t^{(n,A)}$ and $\nu^f(u)= t^{gf}$. If an arbitrarily given element $u^iv^j \in M$ is in the kernel of $\nu^f$, then $\nu^f(u^i) = \nu^f(v^{-j}) \in  \nu^f(\gpres{u}) \cap \nu^f(\gpres{v})$. Working in the finite group $\<t\> \cong \Z_n$, the condition~(\ref{eqn:cyclicCondition}) is equivalent to the assertion that the intersection $\nu^f(\<u\>) \cap \nu^f(\<v\>)= \<t^{gf}\> \cap \<t^{(n,A)}\>$ is trivial. This is because
\begin{alignat*}{1}
  \frac{n/(n,A)}{|\<t^{(n,A)}\> \cap \<t^{gf}\>|} &= \left|\frac{\<t^{(n,A)}\>}{\<t^{(n,A)}\> \cap \<t^{gf}\>}\right|\\
   &= \left|\frac{\<t^{(n,A)}\>\<t^{gf}\>}{\<t^{gf}\>}\right|\\&= \frac{n/(n,A,gf)}{n/(n,gf)}= \frac{(n,gf)}{(n,A,gf)}.
\end{alignat*}
Now the intersection $G \cap \<v\> = G \cap \<t^{(n,A)}\>$ is trivial and so if the condition (\ref{eqn:cyclicCondition}) holds and $u^iv^j \in G \cap M$, then $v^j = \nu^f(v^j) \in \nu^f(\<u\>) \cap \nu^f(\<v\>) = 1$, so $v^j = 1$ and $u^iv^j = u^i \in \gpres{u}$. Conversely, if $G \cap M \leq \gpres{u}$ and $\nu^f(u^i) = \nu^f(v^j) \in \nu^f(\<u\>) \cap \nu^f(\<v\>)$, then $u^{-i}v^{j} \in G \cap M \leq \gpres{u}$ so $v^{j} = 1$ and so $\nu^f(\<u\>) \cap \nu^f(\<v\>) = 1$, whence (\ref{eqn:cyclicCondition}) is satisfied.

In the decomposition~(\ref{eqn:Esplit}) we have $M = \pres{v,u}{v^{m}, u^{\rho}v^\alpha u^{-\sigma}v^{-\alpha}}$ and $v = t^{(n,A)}$, $u = y^g$. Now considering the decomposition $E \cong \<t\> \ast_{\<v\>} M \ast_{\<u\>} \<y\>$ and the fact that $G \cap \<t\> = 1$, it follows from~\cite[I.7.11]{DD89} that each finite subgroup of $G$ is conjugate to a subgroup of $\gpres{y}$, or $G\cap M$ (which, by the above, is contained in $\gpres{u}=\gpres{y^g}$), and hence to a subgroup of $G \cap \<y\>$.
Now Lemma~\ref{lemma:M} implies that $\gpres{y}\cong \Z_{g\mu}$ and the image of the restricted retraction $\nu^f|_{\<y\>}:\<y\> \rightarrow \<t\> \cong \Z_n$ is $\<t^{f}\>$ so $G \cap \<y\> \cong \pres{y^{n/(n,f)}}{y^{g\mu}}\cong \Z_{g\mu / (g\mu,n/(n,f))}$. Now since $f(r-s)\equiv 0$~mod~$n$ we have that $n/(n,f)$ divides $(r-s)$. Also $g\mu=(r-s)(\rho^m-\sigma^m)/(\rho-\sigma)$ so $r-s$, and hence $n/(n,f)$, divides $g\mu$. Therefore $(g\mu,n/(n,f))=n/(n,f)$ so $G \cap \<y\> \cong \Z_{g\mu / (n/(n,f))}$, as required. In particular, if $r^{n/(n,A)} = s^{n/(n,A)}$, then $\mu = 0$ and so $G$ is torsion-free so $G$ has geometric dimension at most two by Theorem~\ref{thm:Esubgroups}(d).
\end{proof}

Note that~(\ref{eqn:cyclicCondition}) is satisfied in the case $f=0$ (ie for the groups $\Gamma=G_n(x_0^rx_A^{-s})$), and note that in the case $(r,s)=1$ condition~(\ref{eqn:cyclicCondition}) is the same as condition~(\ref{eq:cyclicGbar}).

\begin{example}[Finite subgroups of groups in~$\mathfrak{M}$]
\em The condition (\ref{eqn:cyclicCondition}) is satisfied if $(r,n,s,f,A) = (6,24,2,12,8)$, so in this case every finite subgroup of $G$ is finite cyclic with order dividing $26$.
The condition is not satisfied if $(r,n,s,f,A) = (6,24,2,6,9)$. Here $M$ is metacyclic of order $8(3^8-1)$ generated by $v=t^9, u=y^2$ and $\gpres{u}$ has index 8 in~$M$. We have that $\mathrm{ker}(\nu^f|_M) =\mathrm{ker}(\nu^f)\cap M=G\cap M$ has index $8$ in $M$ so $G \cap M$ has order $3^8-1=6560$. Now $\gpres{u}$ is normal in $M$ so $\gpres{u}\cap \mathrm{ker}(\nu^f|_M)=\gpres{u^2}$ has index 2 in $\gpres{u}$ so it follows that $\gpres{u^2}$ has index 2 in $G\cap M$, which is therefore generated by $u^2=y^4$ and an element of $(G\cap M) \backslash \gpres{u^2}$, such as $v^4u=t^{12}y^2$. Since $(v^4u)u^2(v^4u)^{-1} = (v^4uv^{-4})^2 = (u^{3^4})^2 = (u^2)^{81}$ we have that $u^2$ and $v^4u$ do not commute so $G\cap M$ is a nonabelian metacyclic group of order $6560$.
\em
\end{example}

We now consider the occurrence of two-dimensional groups in the class~$\mathfrak{M}$. It follows from Theorem~\ref{thm:Esubgroups}(d) that each group in~$\mathfrak{M}$ possesses a finite index subgroup of geometric dimension at most two and so the torsion-free groups in class~$\mathfrak{M}$ are of geometric dimension at most two. The class~$\mathfrak{M}$ does contain virtually free groups (Corollary~\ref{cor:Gshort}(b)) and it contains free groups such as $G_3(x_0x_1x_2)$, which is free of rank two. We now illustrate that the class~$\mathfrak{M}$ also contains groups that are not virtually free as well as groups with torsion that have torsion-free subgroups that are not free.

\begin{example}[Geometric dimension of groups in~$\mathfrak{M}$]
\em The condition (\ref{eqn:cyclicCondition}) is satisfied when $f=0$. For example, with $(r,n,s,f,A) = (3,2,-3,0,1)$, the group $G =\G \cong G_2(x_0^3x_1^3)=\pres{x_0,x_1}{x_0^3x_1^3}=\gpres{x_0}\ast_{x_0^3=x_1^{-{3}}}\gpres{x_1}$ is a free product of two infinite cyclic groups with subgroups of index $3$ amalgamated. This group is torsion-free and has geometric dimension equal to two, being a nonabelian torsion-free one-relator group with nontrivial center, so it is not free.
The condition (\ref{eqn:cyclicCondition}) is not satisfied if $(r,n,s,f,A) = (3,2,-3,1,1)$.
Here $M$ is generated by $v=t$ and $u=y^3$ and $\gpres{u}$ has index 2 in $M$. We have that $\mathrm{ker}$ $(\nu^f|_M)=\mathrm{ker} (\nu^f)\cap M= G \cap M$ has index 2 in $M$ and that $\gpres{u}$ is normal in $M$ so $\gpres{u}\cap \mathrm{ker}$ $(\nu^f|_M)=\gpres{u^2}$ has index 2 in $\gpres{u}$ so it follows that $\gpres{u^2}$ has index 2 in $G\cap M$, which is therefore generated by $u^2=y^6$ and an element of $(G\cap M)\backslash \gpres{u^2}$, such as $ut=y^3t$. Then $\gpres{y^3t}\cap{y^6}=1$ so $G\cap M$ is an infinite dihedral group $\<y^6\> \rtimes \<y^3t\> \cong D_\infty$ in $E = G \rtimes_\theta \Z_2 \cong \pres{t,y}{t^2,y^3ty^3t^{-1}}$.
In this case, setting $x_i = t^iyt^{-1-i}$, we have $G \cong G_2(x_0x_1x_0^2x_1x_0) = G_2((x_0^2x_1)^2)$ is not torsion-free, because $x_0^2x_1=1$ implies $G\cong \Z_6$, a contradiction. Therefore $G$ has infinite geometric dimension. However $G$ does contain the index two subgroup $G  \cap \G$ of geometric dimension two, which is therefore torsion-free but not free.
\em
\end{example}

Generalizing the previous example, we now wish to identify when the group $G$ possesses a finite index subgroup of geometric dimension exactly two. By Corollary~\ref{cor:Gshort}(b), a necessary condition is that $r^{n/(n,A)} = s^{n/(n,A)} \neq 0$, which means that either  $r = s \neq 0$ or else $r = -s \neq 0$ and $n/(n,A)$ is even. In particular, $g = (r,s) = r \neq 0$.

\begin{theorem}[Two-Dimensional Groups in $\mathfrak{M}$] \label{thm:2dim} With the notation (\ref{eqn:parameters}), assume that $r^{n/(n,A)} = s^{n/(n,A)} \neq 0$.
\begin{enumerate}
\item[(a)] If $\dis \frac{(n,A)}{n} + \frac{1}{(r,s)} > 1$, then $\G$ is free of rank $(n,A)$ and has geometric dimension one, so $G$ and $E$ are virtually free.
\item[(b)] If $\dis \frac{(n,A)}{n} + \frac{1}{(r,s)} \leq 1$, then $\G$ has geometric dimension two, so both $G$ and $E$ possess non-free finite index subgroups of geometric dimension two.
\item[(c)] The group $G$ is torsion-free and satisfies $\gd (G) = \gd (\G) \leq 2$ unless $n/(n,A)$ is even, $r = -s \neq 0$, $(r,s)f \equiv n/2$~mod~$n$, and $t^{n/2}$ is an odd power of $t^{(n,A)}$ in $\<t\> \cong \Z_n$, in which case $G$ contains the infinite dihedral group $\<y^{2r}\> \rtimes \<y^r t^{n/2}\> \cong \Z \rtimes \Z_2$.
\end{enumerate}
\end{theorem}

\begin{proof}
As in the final step of the proof of Theorem~\ref{thm:Esubgroups}(d), the index $n$ subgroup $\G = \npres{y}_E$ is isomorphic to the free product of $(n,A)$ copies of the group $G_{m}(x_0^rx_1^{-s}) = G_{m}(x_0^gx_1^{\pm g})$. If either $(n,A)=n$ or $g = 1$, then the group $G_{m}(x_0^rx_1^{-s})$ is infinite cyclic, so that $\G$ is free of rank $(n,A)$, as in (a). If both $m$ and $g$ are at least two, then (as in the proof of Theorem~\ref{thm:Esubgroups}) $G_{m}(x_0^rx_1^{-s})$ is the free product of $m$ infinite cyclic groups with a nontrivial proper common subgroup of index $g$ amalgamated. It follows that $G_{m}(x_0^rx_1^{-s})$ is not free, being nonabelian with nontrivial center (see e.~g.~\cite[Cor.~4.5]{MKS76}). Thus $\G$ is two-dimensional, which proves (b).

For the final claim, the group $G$ is a subgroup of $E = \<t\> \ast_{t^{(n,A)} = v} M \ast_{u = y^g} \<y\>$ where $M$ is the semidirect product $M = \pres{u,v}{v^{m},uv^{\alpha}u^{-\epsilon} v^{-\alpha}} \cong \Z \rtimes_\epsilon \Z_{m}$. Here, $\epsilon = r/s = \pm 1$ and $m$ is even if $\epsilon = -1$. The group $G$ is the kernel of the retraction $\nu = \nu^f: E \rightarrow \<t\> \cong \Z_n$ where $\nu(t) = t$ and $\nu(y) = t^f$. Thus $\nu(v) = v = t^{(n,A)}$ and $\nu(u) = t^{gf}$. By the subgroup theorem for amalgamated products~\cite{KarrassSolitar70} (or the structure theorem for groups acting on trees~\cite[I.4.1]{DD89}), $G$ is the fundamental group of a graph of groups where the vertex groups are conjugates of the intersections $G \cap \<t\>$, $G \cap M$, and $G \cap \<y\>$. Moreover, any element of finite order in $G$ is conjugate to an element of a vertex group~\cite[I.4.9]{DD89}. Now $G \cap \<t\>$ is trivial and $G \cap \<y\>$ is infinite cyclic. In addition, each torsion-free subgroup of $M \cong \Z \rtimes_\epsilon \Z_{m}$ is infinite cyclic. This means that if $G$ is torsion-free, then $G$ is the fundamental group of a graph of groups where each vertex group is either trivial or infinite cyclic. As noted previously, this in turn implies that $\gd (G) \leq 2$~\cite[Theorem~4.6]{CCH81} and, since $G\cap \Gamma$ is a finite index subgroup of $G$ and of $\Gamma$, it follows that $\gd (G) = \cd (G) = \cd (G) \cap \G = \cd (\G) = \gd (\G)$~\cite[IV.3.18]{DD89}.

It remains to determine the conditions under which $G \cap M$ is not torsion-free. First note that if $r = s$, so that $\epsilon = 1$ and $M \cong \Z \times \Z_{m} = \<u\> \times \<v\>$, then all torsion elements of $M$ lie in $\<v\>$ and so it follows that $G \cap M$ is torsion-free because $\nu|_{\<v\>}$ is injective. So now assume that $r = -s = g \neq 0$ {and $m$ is even}. The fact that $\nu = \nu^f: E \rightarrow \Z_n$ is a retraction implies the condition $f(r-s) \equiv 0$~mod~$n$, which here means that $2gf \equiv 0$~mod~$n$. Now this implies that $\nu(u) = t^{gf}$ is either trivial, so that $u \in G \cap M$, or else $t^{gf} = t^{n/2}$, in which case $\nu(u^2) = 1$. If $\nu(u) = 1$, then the fact that $\nu|_{\<v\>}$ is injective implies that $G \cap M = G \cap (\<u\> \rtimes \<v\>) = \<u\>$, which is infinite cyclic. Thus the group $G$ is torsion-free unless $\nu(u) = t^{gf} = t^{n/2}$ and $G \cap M$ contains a nontrivial element of finite order. The elements of finite order in $M$ have the form $u^iv^j$ where $i=0$ or $j$ is odd. Thus every nontrivial element of finite order in $G \cap M$ has the form $u^iv^j$ where $i \neq 0$ and $0<j<m$ is odd. Since $\nu|_{\<v\>}$ is injective and $v^j \neq 1$, this implies that $1 = \nu(u^iv^j) = t^{gfi}v^j = t^{in/2}v^j = t^{n/2}v^j$ and hence $v^j = t^{n/2}$. In other words, $t^{n/2}$ is an odd power of $v = t^{(n,A)}$ in $\<t\> \cong \Z_n$. Now we find that $u^2 = y^{2g}$ and $uv^j = y^gt^{n/2}$ lie in $G \cap M$ and since $\gpres{u}\cong \Z$, $\gpres{uv^j}\cong \Z_2$ and $(u^2(uv^j)^{-1})^2=1$ we have that $G$ contains the infinite dihedral group $\gpres{u^2}\rtimes \gpres{uv^j}=\<y^{2g}\> \rtimes \<y^gt^{n/2}\> \cong \Z \rtimes \Z_2$.
\end{proof}

For the proof of the following, note that virtually free groups are either large or virtually abelian and that the free product of two nontrivial cyclic groups is large unless both cyclic groups are of order~2~\cite{Pride80} and so the free product is virtually cyclic.

\begin{corollary}[The Tits alternative]\label{cor:Tits}
Let $E$ and $G$ be as defined at~(\ref{eqn:parameters}). Then $E$ (and hence $G$) is either large or is virtually abelian.
\end{corollary}

\begin{proof}
If $A\equiv 0$~mod~$n$ then $E\cong \Z_n*\Z_{|r-s|}$ and if $rs=0$ then $E\cong \Z_n*\Z_{|r|}$ or $\Z_n*\Z_{|s|}$ so we may assume that $A \not \equiv 0$~mod~$n$, hence $m\geq 2$, and that $rs\neq 0$. By Theorem~\ref{thm:Esubgroups}(b)
we may assume $r^m=s^m$ and then by Theorem~\ref{thm:2dim}(a) we may assume $g\geq 2$.
By killing $y^g$, we see that $E$ maps onto $\Z_n*\Z_g$ so we may assume $g=n=2$, and therefore $A=1$. Then $E=\pres{t,y}{t^2, y^2ty^{\pm 2}t}$, which is a $\Z_2$-extension of $G_2(x_0x_1x_0^{-1}x_1^{-1})\cong \Z\times \Z$ or of
$G_2(x_0^2x_1^2)=\pres{x_0,x_1}{x_0^2x_1^2}$, which is the (virtually $\Z\times \Z$) fundamental group of the Klein bottle.
\end{proof}

\section{Free products and Metacyclic Groups in the Class~$\mathfrak{M}$}\label{sec:freeprodmeta}

We now prove Theorem~\ref{mainthm:Mgroups}.

\begin{proof}[Proof of Theorem~\ref{mainthm:Mgroups}]
The group $E$ acts on the acts on the standard graph $T$ associated to the amalgamated free product decomposition~(\ref{eqn:Esplit}). When $g=(r,s)=1$ this decomposition takes the simplified form $E\cong \gpres{t} \ast_{t^{(n,A)}=v} M$. As in \cite[page 13]{DD89}, the standard graph has vertices and edges determined by coset spaces:
$$
\mathrm{Vert}(T) = E/\<t\> \sqcup E/M, \quad  \mathrm{Edge}(T) = E/\<v\>.
$$
Given $w \in E$, the edge $w\<v\> \in \mathrm{Edge}(T)$ has endpoints $w\<t\>$ and $wM$. The action of $E$ is by left multiplication on cosets and the normal form theorem for free products with amalgamation amounts to the fact that $T$ is a tree \cite[I.7.6]{DD89}. The $E$-action on $T$ restricts to an action by the cyclically presented group $G$ and the Structure Theorem for groups acting on trees \cite[I.4.1]{DD89} describes a splitting of $G$ as the fundamental group of a graph of groups where the underlying graph is the orbit graph $G\backslash T$. In particular, since the intersection $G \cap \gpres{t}=1$ is trivial, in the present situation the group $G$ acts freely on the edges of $T$ and the Structure Theorem implies that $G$ is the free product of the vertex groups together with the fundamental group $\pi_1(G \backslash T)$ of the orbit graph, which is a free group.

Consider the orbit map $q: T \ra G \backslash T$. Taking $H = \gpres{t}, \gpres{v}$, and $M$ in turn, we must consider the $G$-action on the coset space $E/H$. The set of orbits of the $G$-action on $E/H$ is the double coset space $G \backslash E/H$, which is in bijective correspondence with the finite coset space $\<t\>/\nu^f(H)$. Thus the orbit graph is finite. Moreover it follows that the group $G$ acts transitively on the vertices of $T$ of the form $w\gpres{t}$. Similarly, the number of vertices in $G \backslash T$ of the form $q(wM)$ is $|\gpres{t}/\nu^f(M)| = |\gpres{t}/\gpres{t^{(n,A)},t^f}| = (n,A,f)$. The number of edges in the orbit graph $G \backslash T$ is $|\gpres{t}/\nu^f(\gpres{v})| = |\gpres{t}/\nu^f(\gpres{t^{(n,A)}})| = (n,A)$. Thus the Euler characteristic of the connected orbit graph is $\chi(G \backslash T) = 1+(n,A,f) - (n,A)$ and so its fundamental group is free of rank $1 - \chi(G \backslash T) = (n,A) - (n,A,f)$.

We now describe the $G$-stabilizers of representatives for each $G$-orbit on the elements of $T$. We have already noted that $\mathrm{Stab}_G(1\gpres{t}) = G \cap \gpres{t} = 1$ and it follows that $\mathrm{Stab}_G(w\gpres{v}) = 1$ for each $w \in E$. For $w \in E$, the $G$-stabilizer of the vertex $wM \in \mathrm{Vert}(T)$ is $G \cap wMw^{-1} = w(G \cap M)w^{-1}$. In the graph of groups decomposition of $G$, there are thus $(n,A,f)$ vertex groups isomorphic to the intersection $\bar{G} = \mathrm{ker}\nu^f|_M=G \cap M$. It remains to show that $\bar{G}$ has the presentation given in the statement.

We set $\bar{\Gamma}= \mathrm{ker}\nu^0|_M = \Gamma \cap M$. Since $(\alpha,m)=1$ and $(r,s)=1$ Lemma~\ref{lem:Pridescyclic} gives that the group $\bar{\Gamma}$ is cyclic of order $|r^m-s^m|$ generated by $u$ so has the presentation $B(r^m-s^m,1,1,1)$. This provides the required presentation in the case $f=0$ so now assume $f\neq 0$.

Let $N=\bar{\Gamma} \cap \bar{G}= (\Gamma \cap G) \cap M$; then $N$ is normal in $M$ and is cyclic, being a subgroup of $\bar{\Gamma}$. Further, $\bar{G}/N=\bar{G}/(\bar{\Gamma}\cap \bar{G})\cong (\bar{\Gamma}\cdot \bar{G})/\bar{\Gamma}$ so $\bar{G}/N$ is a subgroup of $M/\bar{\Gamma}$ so is cyclic. Now
\[u^j \in N \Leftrightarrow u^j\in\bar{G}\Leftrightarrow \nu^f|_M (u^j)=1 \Leftrightarrow t^{jf}=1 \Leftrightarrow n| jf \Leftrightarrow n/(n,f)|j\]
so $N=\pres{u^{n/(n,f)}}{u^{r^m-s^m}}\cong \Z_{|r^m-s^m|/(n/(n,f),r^m-s^m)}$. But $f(r-s)\equiv 0$~mod~$n$ implies that $n/(n,f)$ divides $r^m-s^m$ so $(n/(n,f),r^m-s^m)=n/(n,f)$ and $N\cong \Z_{|r^m-s^m|(n,f)/n}$. Also $M/\bar{G} \cong \mathrm{Im}(\nu^f|_M) = \pres{t^{(n,A)},t^f}{t^n}=\pres{t^{(n,A,f)}}{t^n}\cong \Z_{n/(n,A,f)}$ so $\bar{G}$ has index $n/(n,A,f)$ in $M$. Now Lemma~\ref{lemma:M} implies that $M$ has order $m(r^m-s^m)$ so $\bar{G}$ has order $|r^m-s^m|(n,A,f)/(n,A)$. Hence $\bar{G}/N$ is cyclic of order $\bar{G}/|N|=n(n,A,f)/((n,A)(n,f))$ and so $\bar{G}$ admits a metacyclic extension
\[\Z_{\frac{|r^m-s^m|(n,f)}{n}}\hookrightarrow \bar{G} \twoheadrightarrow \Z_{\frac{n(n,A,f)}{(n,A)(n,f)}} \]
so it has a presentation of the form $B( \frac{|r^m-s^m|(n,f)}{n}, \frac{n(n,A,f)}{(n,A)(n,f)}, R, \lambda)$ for some $R,\lambda$.

Writing $d=(n,f)$, $\delta=(n,f)/(n,A,f)$ we shall show that $G$ is isomorphic to $B(d(r^m-s^m)/n,m/\delta,p,\lambda)$ where
\[\lambda =\frac{Q}{r^m-s^m}\left( \frac{d(r^m-s^m)}{n}, p-1\right)\]
where  $p=(s\hat{r})^{f\bar{\alpha}/(n,A,f)}$, $Q=\frac{(n,A)}{(n,A,f)}\sum_{i=0}^{m/\delta-1} p^i$. Then since the finite group $G=\mathrm{ker}\nu^f|_M$ has a balanced presentation, and hence $H_2(\bar{G})=0$ (so $\bar{G}$ is a Schur group), it follows from~\cite[page~414]{Beyl73} that
\[B(d(r^m-s^m)/n,m/\delta,p,\lambda)\cong B(d(r^m-s^m)/n,m/\delta,p,1)\]
as required. The group $B(d(r^m-s^m)/n,m/\delta,p,\lambda)$ has a presentation
\begin{alignat}{1}
\pres{a,b}{a^{d(r^m-s^m)/n},bab^{-1}=a^p, b^{m/\delta}=a^{dQ/n}}.\label{eq:BillsBeyl}
\end{alignat}
The hypothesis $f(r-s)\equiv 0$~mod~$n$ implies that $n/d$ is a divisor of $(r^m-s^m)$  and so the exponent ${d(r^m-s^m)/n}$ in this presentation is an integer; the exponent $m/\delta = \frac{(n,fn/(n,A))}{(n,f)}$ so too is an integer. That the exponent ${dQ/n}$ is an also an integer, or equivalently that $Q\equiv 0$~mod~$n/d$, is less obvious, so we now explain this. The hypothesis $f(r-s)\equiv 0$~mod~$n$ implies that $r\equiv s$~mod~$n/d$. Since also $r\hat{r}\equiv 1$~mod~$(r^m-s^m)$ and $n/d$ is a divisor of $(r^m-s^m)$ we have $r\hat{r}\equiv 1$~mod~$n/d$. Therefore $p=(s\hat{r})^{f\bar{\alpha}/(n,A,f)}\equiv (r\hat{r})^{f\bar{\alpha}/(n,A,f)}\equiv 1$~mod~$n/d$ and hence
\[Q=\frac{(n,A)}{(n,A,f)}\sum_{i=0}^{m/\delta-1}p^i\equiv \frac{(n,A)}{(n,A,f)}\cdot \frac{m}{\delta} \equiv 0~\mathrm{mod}~n/d.\]
Let $\zeta=u^{n/d}$, $\eta=u^{(n,A)/(n,A,f)}v^{-f/(n,A,f)}$. Then $\zeta\in \mathrm{ker}\nu^0|_M\cap \mathrm{ker}\nu^f|_M=N$ and $\eta\in \mathrm{ker}\nu^f|_M=\bar{G}$.
We now show that the relations of the presentation~(\ref{eq:BillsBeyl}) are satisfied for $\zeta,\eta$ in $M$, and hence in $\bar{G}$. Of course $\zeta^{d(r^m-s^m)/n}=u^{r^m-s^m}=1$, giving the first relation. Next note that the relator $u^rv^{\alpha}u^{-s}v^{-\alpha}$ of $M$ implies that $v^{-\alpha}u^rv^\alpha=u^s$ so
\begin{alignat*}{1}
v^{-f/(n,A,f)}uv^{f/(n,A,f)}
&=v^{-\alpha f\bar{\alpha}/(n,A,f)}u^{r\hat{r}}v^{\alpha f\bar{\alpha}/(n,A,f)}\\
&=v^{-\alpha (f\bar{\alpha}/(n,A,f)-1)}(v^{-\alpha}u^{r}v^\alpha)^{\hat{r}} v^{\alpha(f\bar{\alpha}/(n,A,f)-1)}\\
&=v^{-\alpha (f\bar{\alpha}/(n,A,f)-1)}u^{\hat{r}s} v^{\alpha(f\bar{\alpha}/(n,A,f)-1)}\\\displaybreak[0]
&=v^{-\alpha (f\bar{\alpha}/(n,A,f)-2)}(v^{-\alpha}u v^\alpha)^{\hat{r}s}v^{\alpha(f\bar{\alpha}/(n,A,f)-2)}\\\displaybreak[0]
&=v^{-\alpha (f\bar{\alpha}/(n,A,f)-2)}(v^{-\alpha}u^{r\hat{r}} v^\alpha)^{\hat{r}s}v^{\alpha(f\bar{\alpha}/(n,A,f)-2)}\\\displaybreak[0]
&=v^{-\alpha (f\bar{\alpha}/(n,A,f)-2)}(v^{-\alpha}u^{r} v^\alpha)^{\hat{r}^2s}v^{\alpha(f\bar{\alpha}/(n,A,f)-2)}\\
&=v^{-\alpha (f\bar{\alpha}/(n,A,f)-2)}u^{(\hat{r}s)^2}v^{\alpha(f\bar{\alpha}/(n,A,f)-2)}\\
&=\cdots\\
&=u^{(\hat{r}s)^{f\bar{\alpha}/(n,A,f)}}=u^p.
\end{alignat*}
Then
\begin{alignat*}{1}
\eta\zeta\eta^{-1}&=(u^{(n,A)/(n,A,f)}v^{-f/(n,A,f)})u^{n/d}(u^{(n,A)/(n,A,f)}v^{-f/(n,A,f)})^{-1}\\
&=u^{(n,A)/(n,A,f)}\left(v^{-f/(n,A,f)}uv^{f/(n,A,f)}\right)^{n/d}u^{-(n,A)/(n,A,f)}\\
&=u^{(n,A)/(n,A,f)}u^{np/d}u^{-(n,A)/(n,A,f)}\\
&=u^{np/d}=\zeta^p
\end{alignat*}
giving the second relation of~(\ref{eq:BillsBeyl}). Next
\begin{alignat*}{1}
\eta^{m/\delta}
&=(u^{(n,A)/(n,A,f)}v^{-f/(n,A,f)})^{m/\delta}\\
&=u^{(n,A)/(n,A,f)}(v^{-f/(n,A,f)}u^{(n,A)/(n,A,f)}v^{f/(n,A,f)})\cdot\\
&\quad \quad (v^{-2f/(n,A,f)}u^{(n,A)/(n,A,f)}v^{2f/(n,A,f)})\cdot \\\displaybreak[0]
&\qquad \qquad \ldots \cdot (v^{-(m/\delta-1)f/(n,A,f)}u^{(n,A)/(n,A,f)}v^{(m/\delta-1)f/(n,A,f)})\cdot v^{-mf/\delta}\\
&=\prod_{i=0}^{m/\delta-1}\left(v^{-if/(n,A,f)} u v^{if/(n,A,f)}\right)^{(n,A)/(n,A,f)}.
\end{alignat*}
Now
\begin{alignat*}{1}
v^{-if/(n,A,f)}uv^{if/(n,A,f)}& = v^{-(i-1)f/(n,A,f)}(v^{-f/(n,A,f)}uv^{f/(n,A,f)})v^{(i-1)f/(n,A,f)}\\
&= v^{-(i-1)f/(n,A,f)}u^pv^{(i-1)f/(n,A,f)}\\
&= (v^{-(i-1)f/(n,A,f)}uv^{(i-1)f/(n,A,f)})^p\\
&= \cdots\\
&= u^{p^i}
\end{alignat*}
so
\[\eta^{m/\delta}=\prod_{i=0}^{m/\delta-1}u^{\frac{(n,A)}{(n,A,f)}p^i}=u^{\frac{(n,A)}{(n,A,f)}\sum_{i=0}^{m/\delta-1}p^i}=u^Q.\]
Therefore $\eta^{m/\delta}=u^Q=\zeta^{dQ/n}$, giving the third relation of~(\ref{eq:BillsBeyl}).

Therefore, if $\bar{H}$ is the group defined by the presentation~(\ref{eq:BillsBeyl}) we have shown that there is a homomorphism $\phi: \bar{H} \rightarrow \bar{G}$ given by $a \mapsto \zeta, b\mapsto \eta$. To complete the proof we must show that $\phi$ is an isomorphism. Now $bab^{-1}=a^p$ so the normal closure of $a$ in $\bar{H}$ is the cyclic group generated by $a$, which is of order at most $d|r^m-s^m|/n$. Further, the quotient of $\bar{H}$ by this normal closure is $\pres{b}{b^{m/\delta}}\cong \Z_{m/\delta}$ so $\bar{H}$ has order at most $(d|r^m-s^m|/n) (m/\delta)=|\bar{G}|$. Therefore, if $\phi$ is surjective then it follows that $\bar{H} \cong \bar{G}$ so it suffices to show that $\bar{G}$ is generated by $\zeta$ and $\eta$. We have ${\eta^l}\in N$ if and only if
\[ \nu^0|_M((u^{(n,A)/(n,A,f)}v^{-f/(n,A,f)})^l)=1\]
and
\[ \nu^f|_M((u^{(n,A)/(n,A,f)}v^{-f/(n,A,f)})^l)=1.\]
That is, if and only if
\[v^{-fl/(n,A,f)}=1 \Leftrightarrow m|\frac{fl}{(n,A,f)} \Leftrightarrow \frac{m}{\delta}| \frac{f}{d} l \Leftrightarrow \frac{m}{\delta}| l\]
since $(m/\delta,f/d)=1$. That is, $(\eta N)\in \bar{G}/N$ has order $m/\delta$ in $\bar{G}/N$ so the cosets $N,\eta N,\ldots , \eta^{m/\delta-1}N$ are distinct. In total there are $|N|m/\delta=|r^m-s^m|(n,A,f)/(n,A)=|\bar{G}|$ elements in these cosets. Since also $N$ is generated by $\zeta$ we have that $\bar{G}$ is generated by $\zeta$ and $\eta$, as required.
\end{proof}

We now turn to the proof of Corollary~\ref{maincor:EWgroupswithC}.
When (C) holds but (A) does not, the following lemma puts the group $G_n(x_0x_kx_l)$ in a form to which Theorem~\ref{mainthm:Mgroups} can be readily applied.

\begin{lemma}\label{lem:CTrueAFalseGpsNew}
Suppose $(n,k,l)=1$, $k\neq l$, $k,l\neq 0$. If (C) holds but (A) does not then $G_n(x_0x_kx_l)$ is isomorphic to $G_n(x_0x_{n/3}x_{1+2n/3})$.
\end{lemma}

\begin{proof}
As in~\cite{EW10} we shall write $G_n(k,l)=G_n(x_0x_kx_l)$. As stated in~\cite[page~765]{EW10} the given conditions imply that $(n,k)=1$ or $(n,l)=1$ or $(n,k-l)=1$. By the symmetries of~\cite[Lemma~2.1(ii),(iv)]{EW10} we may assume $(n,l)=1$ and then by applying an automorphism of $\Z_n$ we may assume $l=1+2n/3$.
Observe that
\begin{alignat}{1}
&\ G_n(1+n/3,1+2n/3) \overset{(ii)}{\cong} G_n(1+2n/3,1+n/3) \overset{(iv)}{\cong} G_n(1+2n/3,n/3) \nonumber\\
&\quad \overset{(ii)}{\cong} G_n(n/3,1+2n/3) \overset{\times -1}{\cong}G_n(2n/3,n/3-1)\overset{(iv)}{\cong} G_n(2n/3,1+n/3) \label{eq:isoms}
\end{alignat}
(where the superscripts correspond to isomorphisms given in~\cite[Lemma~2.1]{EW10}, or multiplying subscripts of the relators by $-1$). In particular
\begin{alignat}{1}
G_n(1+n/3,1+2n/3), G_n(2n/3,n/3-1), G_n(2n/3,n/3+1)\label{eq:isomgps}
\end{alignat}
are all isomorphic to $G_n(n/3,2n/3+1)$.

Since~(C) is true we have that $3k\equiv 0$~mod~$n$ or $3(l-k)\equiv 0$~mod~$n$ so for $G=G_n(k,1+2n/3)$ we have $k=n/3,2n/3,1,$ or $n/3+1$.
If $k=n/3$ then we are done. If $k=n/3+1$ then $G=G_n(1+n/3,1+2n/3)$ and, as noted at~(\ref{eq:isomgps}), $G\cong G_n(n/3,1+2n/3)$.
Thus we need to consider the cases $k=2n/3$, $k=1$. These are the same case, however, since
\[G_n(2n/3,1+2n/3)\overset{(ii)}{\cong} G_n(1+2n/3,2n/3) \overset{(iv)}{\cong} G_n(1+2n/3,1) \overset{(ii)}{\cong} G_n(1,1+2n/3)\]
so we may assume $k=2n/3$, so $G=G_n(2n/3,1+2n/3)$. Now since (A) does not hold we have that $2n/3 \equiv 0,2$~mod~$3$.

Suppose $2n/3\equiv 0$~mod~$3$; then $9|n$. Then $(1+2n/3,n)=1$ so the map $x_i\mapsto x_{(1+2n/3)i}$ is an automorphism of $G_n(2n/3,1+2n/3)$. This transforms the relators $x_ix_{i+2n/3}x_{i+2n/3+1}$ to $x_ix_{i+4n^2/9+2n/3}x_{i+4n^2/9+4n/3+1}$, ie to $x_ix_{i+2n/3}x_{i+1+n/3}$ which are the relators of $G_n(2n/3,1+n/3)$ which, as observed at (\ref{eq:isoms}), is isomorphic to $G_n(n/3,1+2n/3)$.

Suppose then that $2n/3\equiv 2$~mod~$3$; then $9|(2n-6)$. Then $(2n/3-1,n)=1$ so the map $x_i\mapsto x_{(2n/3-1)i}$ is an automorphism of $G_n(2n/3,1+2n/3)$. This transforms the relators $x_ix_{i+2n/3}x_{i+2n/3+1}$ to $x_ix_{i+(2n/3)(2n/3-1)}x_{i+(2n/3+1)(2n/3-1)}$, ie to $x_ix_{i+(2n-6)2n/9+2n/3}x_{i+(2n-6)2n/9+4n/3-1}$, ie $x_ix_{i+2n/3}x_{i+n/3-1}$, which are the relators of $G_n(2n/3,n/3-1)$, which as observed at (\ref{eq:isoms}), is isomorphic to $G_n(n/3,1+2n/3)$.
\end{proof}

\begin{proof}[Proof of Corollary~\ref{maincor:EWgroupswithC}]
Let $G=G_n(x_0x_kx_l)$ where $(n,k,l)=1$ and suppose that~(C) holds.

\noindent (i) Suppose that~(A) does not hold and $k \not \equiv 0, l \not \equiv 0, k \not \equiv l$~mod~$n$. Then Lemma~\ref{lem:CTrueAFalseGpsNew} implies that $G$ is isomorphic to $G_n(x_0x_{n/3}x_{1+2n/3})$ which is the group $G_n(w)$ with $w$ as given at~(\ref{eqn:OurCyclicw}) with parameters $(r,n,s,f,A)=(2,n,-1,n/3,1)$. The hypotheses of Theorem~\ref{mainthm:Mgroups} are satisfied
and we have $r^n-s^n=2^n-(-1)^n$, $(n,f)=n/3$, $\bar{\alpha}=1$, $\hat{r}=(-1)^n 2^{n-1}$ so  $G\cong B((2^n-(-1)^n)/3,3,p,1)$, where $p=2^{n(n-1)/3}$. It remains to show that $p$ is congruent to $2^{2n/3}$ mod~$(2^n-(-1)^n)/3$.
Working to this base, we have $2^n-(-1)^n\equiv 0$ so $2^n\equiv (-1)^n$ and hence $2^{n(n/3-1)}\equiv (-1)^{n(n/3-1)}\equiv 1$. Therefore $2^{n(n-1)/3}=2^{n(n/3-1)}2^{2n/3}\equiv 2^{2n/3}$, and we are done.

\noindent (ii) Suppose that $k \equiv 0, l \equiv 0$ or $k \equiv l$~mod~$n$. Then since $(n,k,l)=1$ it is easy to see that $G\cong G_n(x_0^2x_1)$ which is the group $G_n(w)$ with $w$ as given at~(\ref{eqn:OurCyclicw}) with parameters $(r,n,s,f,A)=(2,n,-1,0,1)$ so by Theorem~\ref{mainthm:Mgroups} we have that $G\cong B(2^n-(-1)^n,n,1,1)\cong \Z_{2^n-(-1)^n}$.

\noindent (iii) Suppose that~(C) holds and~(A) does not. Then, as explained in~\cite[page~763]{EW10}, we have that $G\cong G_{3m}(x_0x_1x_m)$ or $G\cong G_{3m}(x_0x_1x_{m+1})$ ($m\geq 1$). Now we have that $G_{3m}(x_0x_1x_m)\cong G_{3m}(x_0^{-1}x_m^{-1}x_1^{-1})\cong G_{3m}(x_0x_mx_1)$ and that $G_{3m}(x_0x_1x_{m+1})\cong G_{3m}(x_1x_{m+1}x_0)\cong G_{3m}(x_0x_{m}x_{3m-1})$, so $G\cong G_{3m}(x_0x_mx_{\epsilon})$ ($\epsilon=\pm 1$).
Setting $n=3m$, $r=2$, $s=-1$, $f=m$, $A=m+\epsilon$ and comparing the defining relator $x_0x_mx_{\epsilon}$ with the word shape~(\ref{eqn:OurCyclicw}) we see that $G$ lies in the class $\mathfrak{M}$. Since (A) holds we have that $A=m+\epsilon\equiv 0$~mod~$3$ so we write $A=3j$. Since $(m,j)=(3j-\epsilon,j)=1$ we have that $(n,A)=(3m,3j)=3$ and $(n,A,gf)=(3m,3j,m)=1$. In addition we have $n(n,A,gf)=3m=(n,A)(n,gf)$ as in condition~(\ref{eqn:cyclicCondition}). By Theorem~\ref{mainthm:Mgroups} it follows that $G \cong \Z_\gamma \ast F_k$ where $\gamma = \mu (n,A,f)/(n,A) = (2^m-(-1)^m)/3$ and $k=(n,A)-(n,A,f) = 3-1 = 2$.
\end{proof}

\section{Fixed Points for the Shift on Finite Groups $G$ in the Class $\mathfrak{M}$}\label{sec:FixedPoints}

We now prove Theorem~\ref{thm:fpf}.

\begin{proof}[Proof of Theorem~\ref{thm:fpf}]
Since $G$ is finite and nontrivial, Corollary~\ref{cor:Gshort} implies that  $(n,A) = (r,s) = 1$. Both $G$ and $\Gamma$ (see~(\ref{eqn:parameters})) are left $\Z_n$-sets under the action of their respective shifts, which we denote by $\theta_G$ and $\theta_\Gamma$ respectively. By~\cite[Lemma 2.2]{Bogley14}, both $G$ and $\Gamma$ are isomorphic as $\Z_n=\gpres{t}$-sets to the coset space $E/\gpres{t}$ with its natural left action via multiplication in $E$. Thus $G$ and $\Gamma$ are isomorphic as $\Z_n$-sets, and so given $j$ modulo~$n$, there is a bijective correspondence between the fixed point sets $\mathrm{Fix}(\theta_G^j)$ and $\mathrm{Fix}(\theta_\Gamma^j)$. (Note that although these fixed point sets are subgroups within $E$, the bijection referred to here is not generally a group homomorphism.) Next, the fact that $(n,A) = 1$ implies that $\mathrm{Fix}(\theta_\Gamma^j) = \mathrm{Fix}(\theta_\Gamma^{jA})$ and that there is a group isomorphism $\phi: \Gamma_0 = G_n(x_0^rx_1^{-s}) \rightarrow G_n(x_0^rx_A^{-s}) =  \Gamma$ given by $\phi(x_i) = x_{iA}$ and satisfying $\phi \circ \theta_{\Gamma_0} = \theta_\Gamma^A \circ \phi$. Taken together, these observations imply that we have (setwise) bijections and equalities as follows:
$$
\mathrm{Fix}(\theta_G^j) \simeq \mathrm{Fix}(\theta_\Gamma^j) = \mathrm{Fix}(\theta_\Gamma^{jA}) \simeq \mathrm{Fix}(\theta_{\Gamma_0}^j).
$$
Since $(r,s)=1$, Lemma~\ref{lem:Pridescyclic} implies that $\Gamma_0$ is cyclic of order $\mu = |r^n-s^n|$,  generated by any of its generators $x_i$; further, we can choose integers $a,b$ such that $ar^n + bs = 1$. Setting $\beta = r(s^{n-1}a+b)$, the proof of Lemma~\ref{lem:Pridescyclic} (see~\cite[Lemma 3.4]{BW15}) shows that $\theta(x_i) = x_{i+1} = x_i^\beta$. Working modulo $\mu = |r^n-s^n|$, if $x_i^k \in \Gamma_0$ then $\theta^j(x_i^k) = x_i^k$ if and only if $k$ is equivalent to $\beta^j k$ i.e.\,to
\[ r^j(s^{n-1}a+b)^jk \equiv r^j\left(s^{n-1}a+\frac{1-ar^n}{s}\right)^jk \equiv \frac{r^j}{s^j}\left(s^{n}a+1-ar^n\right)^jk
\equiv \frac{r^j}{s^j}k
\]
which is equivalent to the assertion that $(r^j-s^j)k \equiv 0$~mod~$\mu$. Thus, considering the shift $\theta_{\Gamma_0}$ on $\Gamma_0 = G_n(x_0^rx_1^{-s})$, the fixed point subgroup for $\theta_{\Gamma_0}^j$ is $\mathrm{Fix}(\theta_{\Gamma_0}^j) = \gpres{x_0^{\mu/(r^n-s^n, r^j-s^j)} }= \gpres{x_0^{\mu/(r^{(n,j)}-s^{(n,j)})}}$, which has order $|r^{(n,j)}-s^{(n,j)}|$. Translating to $G = G_n(w)$, the fixed point subgroup for $\theta_G^j$ has the indicated order.

If $\theta_G^j$ is fixed point free then we have that $|r^{(n,j)}-s^{(n,j)}|=1$ so $(n,j)=1$ and $|r-s|=1$ and so the only value of $f$ that satisfies $f(r-s) \equiv 0$~mod~$n$ is $f \equiv 0$~mod~$n$ so $G \cong G_n(x_0^rx_A^{-s})$. Since $G$ is finite we have that $(A,n)=1$ so  $G \cong G_n(x_0^rx_1^{-s}) \cong \Z_\mu$ where $\mu = |r^n-s^n|$.
Conversely, if $(n,j)=|r-s|=1$ then $|\mathrm{Fix}(\theta_G^j)|=1$ so $\theta_G^j$ is fixed point free.
\end{proof}

In the next example we use Theorem~\ref{thm:fpf} to show that every group $G$ in the class $\mathfrak{M}$ of order 125 is cyclic. The significance of this will be made clear in the subsequent discussion.

\begin{example}[Groups in~$\mathfrak{M}$ of order 125 are cyclic]\label{ex:125}
\em
Let $G$ be a group in the class $\mathfrak{M}$ with parameters $(r,n,s,f,A)$ and order~125. We shall show that $G$ is cyclic and the shift automorphism $\theta$ is either fixed point free or is the identity map. Let $H=\mathrm{Fix}(\theta)\leq G$. If $H=1$ then $G$ is cyclic, by Theorem~\ref{thm:fpf}; if $H=G$ then $\theta$ is the identity map so all generators of the cyclic presentation for $G$ represent the same element so $G$ is cyclic. We now show that these are the only possibilities; that is, $H$ cannot be a proper, nontrivial subgroup of $G$. We do this by showing that if $|H|=5$ or $25$ then $r$ and $s$ are both divisible by 5, which contradicts $(r,s)=1$ (which holds since $G$ is finite).

Suppose $|H|=5$. Then $|r-s|=5$ (by Theorem~\ref{thm:fpf}) so $r=s+5\epsilon$ ($\epsilon = \pm 1$) and
\[125= |G|=|(s+5\epsilon)^n-s^n| = \big| 5 \epsilon ns^{n-1} + \sum_{j=2}^{n} \binom{n}{j}(5\epsilon)^js^{n-j}\big|\]
or equivalently
\begin{equation}
  5\epsilon ns^{n-1} = \pm 125 - \sum_{j=2}^{n} \binom{n}{j}(5\epsilon)^js^{n-j}.\label{eq:binom125}
\end{equation}
We may assume $s$ is not divisible by $5$, for otherwise $5|(r,s)=1$ so~(\ref{eq:binom125}) implies that $n$ is divisible by 5. Therefore $5| \binom{n}{2}$ so~(\ref{eq:binom125}) implies that $n$ is divisible by 25. Therefore $125=|r^n-s^n|$ is divisible by $|r^{25}-s^{25}|=|(s+5\epsilon)^{25}-s^{25}|$, and for each $\epsilon = \pm 1$ the minimum value of $|(s+5\epsilon)^{25}-s^{25}|$ is greater than 25, and we have a contradiction.

Suppose then that $|H|=25$. Then $|r-s|=25$ and as above we get
\begin{equation}
  25\epsilon ns^{n-1} = \pm 125 - \sum_{j=2}^{n} \binom{n}{j}(25\epsilon)^js^{n-j}\label{eq:binom25}
\end{equation}
so again $5|n$. Then $125=|r^n-s^n|$ is divisible by $|r^5-s^5|=|(s+25\epsilon)^5-s^5|$, but the minimum value of $|(s+25\epsilon)^5-s^5|$ is again greater than 125, so we have a contradiction.
\em
\end{example}

In this article we have investigated structural properties of both finite and infinite cyclically presented groups in the class~$\mathfrak{M}$. As noted in the Introduction, the class~$\mathfrak{M}$ encompasses many families of finite metacyclic generalized Fibonacci groups that were previously identified in the literature. Notable exceptions to this are the following families of groups:
\begin{alignat*}{1}
H(2k+1,4,2) &= G_4\left(\left(\prod_{i=0}^{2k} x_i\right)\left( x_{2k+1}x_{2k+2}\right)^{-1}\right),\\
F(4l+2,4)&=G_4\left( \left( \prod_{i=0}^{4l+1}x_i \right)x_{4l+2}^{-1}\right).
\end{alignat*}
The groups $H(2k+1,4,2)$ were conjectured to be finite and metacyclic in~\cite{CR75LMS} and this was proved in~\cite{Brunner74}. (The paper~\cite{Brunner74} also gives a formula for the group orders but this, and the last few lines of the proof, are incorrect.) The groups $F(4l+2,4)$ were shown in~\cite{Seal82} to be metabelian groups of order $(4l+1)(2^{4l+1}+(-1)^l2^{2l+1}+1)$ with $F(4l+2,4)^\mathrm{ab}\cong \Z_5\times \Z_{4l+1}$. It was later established in~\cite{Thomas90} that $F(4l+2,4)$ and $H(4l+3,4,2)$ are isomorphic (therefore providing the orders of $H(4l+3,4,2)$). The orders of the groups $H(2k+1,4,2)$, and the fact that they are metacyclic, are contained in~\cite[Corollary~E]{BW15}, which provides that if $n = 4$ or $6$, $(m,k) = 1$, and $m = \pm 1$~mod~$n$, then the group $J_n(m,k) = \pres{t,y}{t^n,y^{m-k}t^3y^kt^2}$ contains a unique normal, metacyclic (and cyclically presented) subgroup of index~$n$ and known order. In particular, considering the shift extension we have $H(2k+1,4,2) \rtimes_\theta \Z_4 = J_4(2k-1,-2)$ and $|H(2k+1,4,2)|=(2k-1)(2^{2k-1}-(-1)^{(k+1)(k+2)/2}2^k+1)$.

There is an overlap between the cyclically presented groups in the class~$\mathfrak{M}$ and those covered by~\cite[Corollary~E]{BW15}, as well as between their shift extensions $E(r,n,s,A) = \pres{t,y}{t^n,y^rt^Ay^{-s}t^{-A}}$ and the groups $J_n(m,k)$. For example, $E(1,2,-2,1) \cong J_6(1,1) \cong \Z_6$ and $E(2,4,-1,1) \cong J_4(3,1) \cong J_4(3,-2)$ is a nonabelian group of order $60$, with both $E(2,4,-1,1)$ and $J_4(3,-2)$ serving as the shift extension of $H(5,4,2) \cong \Z_{15}$. However, such coincidences appear to be rare exceptions. The group $J_4(5,-1)=F(6,4) \rtimes_\theta \Z_4$ is metacyclic of order $500$; this essentially appeared in~\cite{CR75LMS}, being a $\Z_4$ extension of $F(6,4)$ of order 125. Elementary arguments show that 500 cannot be expressed in the form $n|r^n-s^n|$ with $n \geq 2$ and $r,s \neq 0$, so $J_4(5,-1)$ is not isomorphic to any group of the form $E(r,n,s,A)$.

Given an arbitrary positive integer $k$, the cyclically presented group in class~$\mathfrak{M}$ corresponding to the parameters $(r,n,s,f,A) = (k+1,2, \pm k,0,1)$ is cyclic of order $(k+1)^2-k^2 = 2k+1$ so every cyclic group of odd order occurs in the class~$\mathfrak{M}$. Considering these two occurrences, when $s = k$ the shift on the group $\Z_{2k+1}$ is fixed point free, while if $s=-k$ then the shift acts as the identity. Thus although the order of the shift extension $|F(6,4) \rtimes \Z_4| = 500$ does not occur as that of the shift extension of any group in the class~$\mathfrak{M}$, the group order $|F(6,4)| = 125$ does occur as the order of a group in the class~$\mathfrak{M}$. However, Example~\ref{ex:125} shows that the only group of order 125 that occurs in the class~$\mathfrak{M}$ is cyclic. The fact that $F(6,4)$ is nonabelian therefore implies that $F(6,4)$ itself does not occur in the class~$\mathfrak{M}$. Alternatively, one can use Corollary~\ref{maincor:MgroupsCor} to prove that no cyclically presented group in~$\mathfrak{M}$ can simultaneously have order 125 and abelianisation $\Z_5\times \Z_5$ (the abelianisation of $F(6,4)$), so the group $F(6,4)$ is not in~$\mathfrak{M}$.

With the evidence so far considered, it remains possible that the only finite groups in the class~$\mathfrak{M}$ that have orders equal to those of the finite generalized Fibonacci groups occurring in~\cite[Corollary~E]{BW15} are finite cyclic. Further, we suspect that the only nontrivial finite cyclic groups that occur in~\cite[Corollary~E]{BW15} are $\Z_5$ (eg $F(2,4)$), $\Z_{15}$ (eg $H(5,4,2)$), or $\Z_{13}$ (eg $H(4,6,3)$). If both of these statements are indeed correct then it would follow that there are precisely three nontrivial finite groups that occur both in the class~$\mathfrak{M}$ and in~\cite[Corollary~E]{BW15}.

\end{document}